\documentclass[11pt,reqno]{amsart}
\usepackage{graphicx}
\usepackage{hyperref}
\usepackage{verbatim}

\newtheorem{thm}{Theorem}[section]
\newtheorem{lem}[thm]{Lemma}
\newtheorem{prop}[thm]{Proposition}
\newtheorem{cor}[thm]{Corollary}

\theoremstyle{definition}

\theoremstyle{remark}

\numberwithin{equation}{section}

\def\R{\mathbb{R}}

\def\ra{\rightarrow}

\def\al{\alpha}

\def\ep{\epsilon}

\def\la{\lambda}

\def\si{\sigma}

\def\om{\omega}

\def\de{\delta}
\def\De{\Delta}

\def\Ga{\Gamma}

\begin{document}

\title[Regularity and stability of transition fronts]{Regularity and stability of transition fronts in nonlocal equations with time heterogeneous ignition nonlinearity}

\author{Wenxian Shen}
\address{Department of Mathematics and Statistics, Auburn University, Auburn, AL 36849}
\email{wenxish@auburn.edu}

\author{Zhongwei Shen}
\address{Department of Mathematics and Statistics, Auburn University, Auburn, AL 36849}
\email{zzs0004@auburn.edu}





\begin{abstract}
The present paper is devoted to the investigation of various properties of transition fronts in nonlocal equations in heterogeneous media of ignition type, whose existence has been established by the authors of the present paper in a previous work. It is first shown that the transition front is continuously differentiable in space with uniformly bounded and uniformly Lipschitz continuous space partial derivative. This is the first time that regularity of transition fronts in nonlocal equations is ever studied. It is then shown that the transition front is uniformly steep. Finally, asymptotic stability, in the sense of exponentially attracting front like initial data, of the transition front is studied.
\end{abstract}

\maketitle

\tableofcontents


\section{Introduction}

Consider
\begin{equation}\label{eqn-nonlocal}
u_{t}=J\ast u-u+f(t,x,u),
\end{equation}
where $J$ is the dispersal kernel and $[J\ast u](x)=\int_{\R}J(x-y)u(y)dy=\int_{\R}J(y)u(x-y)dy$, and the reaction term $f$ is of monostable type, bistable type or ignition type. Such an equation, introduced as a substitute for the classical reaction-diffusion equation
\begin{equation}\label{eqn-classical}
u_{t}=\De u+f(t,x,u),
\end{equation}
has been used to model various diffusive processes with jumps (see e.g. \cite{Fi03} for some background). While a large amount of literature has been carried out to the understanding of \eqref{eqn-classical}, its nonlocal version \eqref{eqn-nonlocal} has attracted a lot of attention recently and some results have been established. For \eqref{eqn-nonlocal} in the homogeneous media, traveling waves, i.e., solutions of the form $u(t,x)=\phi(x-ct)$ with $(c,\phi)$ satisfying
\begin{equation*}
J\ast\phi-\phi+c\phi_{x}+f(\phi)=0,\quad \phi(-\infty)=1,\quad\phi(\infty)=0,
\end{equation*}
have been obtained (see \cite{BaFiReWa97,CaCh04,Ch97,Cov-thesis,CoDu05,CoDu07,Sc80} and references therein). The study of \eqref{eqn-nonlocal} in the heterogeneous media is rather recent and results concerning front propagation are very limited. In \cite{CDM13,ShZh10,ShZh12-1,ShZh12-2}, the authors investigated \eqref{eqn-nonlocal} in the space periodic monostable media and proved the existence of spreading speeds and periodic traveling waves. In \cite{RaShZh}, Rawal, Shen and Zhang studied the existence of spreading speeds and traveling waves of \eqref{eqn-nonlocal} in the space-time periodic monostable media. For \eqref{eqn-nonlocal} in the space heterogeneous monostable media, Berestycki, Coville and Vo studied in \cite{BCV14} the principal eigenvalue, positive solution and long-time behavior of solutions, while Lim and Zlato\v{s} proved in \cite{LiZl14} the existence of transition fronts in the sense of Berestycki-Hamel (see \cite{BeHa07,BeHa12}). In \cite{BeRo}, Berestycki and Rodr\'{i}guez studied \eqref{eqn-nonlocal} with a barrier nonlinearity of monostable type or bistable type, and proved that while propagation always occur in the monostable case, it may be obstructed in the bistable case. For \eqref{eqn-nonlocal} in the time heterogeneous media of ignition type, the authors of the present paper proved in \cite{ShSh14-2} the existence of transition fronts.

In the present paper, we continue to study \eqref{eqn-nonlocal} in the time heterogeneous media based on the work done in \cite{ShSh14-2}. Recall that, an entire solution $u(t,x)$ of \eqref{eqn-nonlocal} is called a \textit{transition front} in the sense of Berestycki-Hamel (see \cite{BeHa07,BeHa12}) if $u(t,-\infty)=1$ and $u(t,\infty)=0$ for any $t\in\R$, and for any $\ep\in(0,1)$ there holds
\begin{equation*}
\sup_{t\in\R}\text{diam}\{x\in\R|\ep\leq u(t,x)\leq 1-\ep\}<\infty.
\end{equation*}
Equivalently, an entire solution $u(t,x)$ of \eqref{eqn-nonlocal} is called a transition front if there exists a function $X:\R\to\R$ such that
\begin{equation*}
\lim_{x\to-\infty}u(t,x+X(t))=1\,\,\text{and}\,\,\lim_{x\to\infty}u(t,x+X(t))=0\,\,\text{uniformly in}\,\,t\in\R.
\end{equation*}
We remark that neither the definition of transition front nor the equation \eqref{eqn-nonlocal} itself guarantees any space regularity of transition fronts beyond continuity. Also, the transition fronts constructed in \cite{LiZl14} and \cite{ShSh14-2} are only uniformly Lipschitz continuous in space; it is not known if they are continuously differentiable in space. One of the main goals of the present paper is to investigate the space regularity of transition fronts constructed in \cite{ShSh14-2}. It should be pointed out that space regularity is of fundamental importance in further studying various important properties, such as uniform steepness and stability, of transition fronts.

Now, let us focus on \eqref{eqn-nonlocal} in the time heterogeneous media of ignition type, i.e.,
\begin{equation}\label{main-eqn}
u_{t}=J\ast u-u+f(t,u),\quad (t,x)\in\R\times\R,
\end{equation}
where the convolution kernel $J$ satisfies
\medskip

\noindent {\bf (H1)} {\it  $J\not\equiv0$, $J\in C^{1}(\R)$, $J(x)=J(-x)\geq0$ for all $x\in\R$, $\int_{\R}J(x)dx=1$, $\int_{\R}|J'(x)|dx<\infty$ and
\begin{equation}\label{decay-convolution-kernel}
\int_{\R}J(x)e^{\la x}dx<\infty,\quad\forall\la>0;
\end{equation}
}

\medskip

\noindent and the time heterogeneous nonlinearity $f(t,u)$ satisfies
\medskip

\noindent {\bf (H2)} {\it  $f:\R\times[0,\infty)\ra\R$ is continuously differentiable and satisfies the following conditions:
\begin{itemize}
\item there are $\theta\in(0,1)$ (the ignition temperature), $f_{\min}\in C^{1,\al}([0,1])$ and a Lipschitz continuous function $f_{\max}:[0,1]\ra\R$ satisfying
\begin{equation*}
\begin{split}
f_{\min}(u)=0=f_{\max}(u),\,\,&u\in[0,\theta]\cup\{1\},\\
0<f_{\min}(u)\leq f_{\max}(u),\,\,&u\in(\theta,1),\\
f_{\min}'(1)<0
\end{split}
\end{equation*}
such that
\begin{equation*}
f_{\min}(u)\leq f(t,u)\leq f_{\max}(u),\quad (t,u)\in[0,1].
\end{equation*}

\item $f(t,u)<0$ for $(t,u)\in\R\times(1,\infty)$

\item first-order partial derivatives are uniformly bounded, i.e.,
\begin{equation*}
\sup_{(t,u)\in\R\times[0,1]}|f_{t}(t,u)|<\infty\quad\text{and}\sup_{(t,u)\in\R\times[0,\infty)}|f_{u}(t,u)|<\infty
\end{equation*}

\item there exists $\tilde{\theta}\in(\theta,1)$ such that $f_{u}(t,u)\leq0$ for all $t\in\R$ and $u\in[\tilde{\theta},1]$.
\end{itemize}
}

\smallskip

For convenience and later use, let us first  summarize the main results obtained in \cite{ShSh14-2}.
To this end, consider the following homogeneous equation
\begin{equation}\label{main-eqn-perturb-homo}
u_{t}=J\ast u-u+f_{\min}(u),\quad (t,x)\in\R\times\R,
\end{equation}
where $f_{\min}$, given in $\rm(H2)$, is of ignition type. Assume (H1) and (H2).  It is proven in \cite{Cov-thesis} that there are a unique $c_{\min}^*>0$ and a unique $C^{1}$ function $\phi=\phi_{\min}:\R\to(0,1)$ satisfying
\begin{equation}\label{tw-homo}
\begin{cases}
J\ast\phi-\phi+c_{\min}^*\phi'+f_{\min}(\phi)=0,\\
\phi'<0,\,\,\phi(0)=\theta,\,\,\phi(-\infty)=1\,\,\text{and}\,\,\phi(\infty)=0.
\end{cases}
\end{equation}
That is, $\phi_{\min}$ is the normalized wave profile and $\phi_{\min}(x-c_{\min}^*t)$ is the traveling wave of \eqref{main-eqn-perturb-homo}. Moreover, using the equation in \eqref{tw-homo}, it is  not hard to see that $\phi_{\min}'$ is uniformly Lipschitz continuous, that is,
\begin{equation}\label{Lip-TW}
\sup_{x\neq y}\bigg|\frac{\phi_{\min}'(x)-\phi_{\min}'(y)}{x-y}\bigg|<\infty.
\end{equation}

The following proposition is proved in \cite{ShSh14-2}.

\begin{prop}[\cite{ShSh14-2}]\label{prop-property-approximating-sol}
Suppose $\rm(H1)$-$\rm(H2)$.
\begin{itemize}
\item[\rm(1)] For $s<0$, there exists a unique $y_{s}\in\R$ with $y_{s}\ra-\infty$ as $s\to-\infty$ such that the classical solution $u(t,x;s)$ of \eqref{main-eqn} with initial data $u(s,x;s)=\phi_{\min}(x-y_{s})$ satisfies the normalization $u(0,0;s)=\theta$ and the following properties:
\begin{itemize}
\item[\rm(i)] $u(t,-\infty;s)=1$, $u(t,\infty;s)=0$ and $u(t,x;s)$ is strictly decreasing in $x$;
\item[\rm(ii)] let $X_{\la}(t;s)$ be such that $u(t,X_{\la}(t;s);s)=\la$ for any $\la\in(0,1)$; there exist $c_{\min}>0$, $c_{\max}>0$,
and  a twice continuously differentiable function $X(\cdot;s):[s,\infty)\to\R$ satisfying
\begin{equation*}
0<c_{\min}\leq\dot{X}(t;s)\leq c_{\max}<\infty,\,\, s<0,\,\,t\geq s\quad\text{and}\quad\sup_{s<0,t\geq s}|\ddot{X}(t;s)|<\infty
\end{equation*}
such that
\begin{equation*}
\forall\la\in(0,1),\quad\sup_{s<0,t\geq s}|X(t;s)-X_{\la}(t;s)|<\infty
\end{equation*}
and there exist exponents $c_{\pm}>0$ and shifts $h_{\pm}>0$ such that
\begin{equation*}
\begin{split}
u(t,x;s)&\geq1-e^{c_{-}(x-X(t;s)+h_{-})}\quad\text{if}\quad x\leq X(t;s)-h_{-},\\
u(t,x;s)&\leq e^{-c_{+}(x-X(t;s)-h_{+})}\quad\text{if}\quad x\geq X(t;s)+h_{+}
\end{split}
\end{equation*}
for all $s<0$, $t\geq s$;

\item[\rm(iii)] $u(t,x;s)$ is uniformly Lipschitz continuous in space, that is,
\begin{equation}\label{locally-uniform-Lip}
\sup_{x\neq y\atop s<0,t\geq s}\bigg|\frac{u(t,y;s)-u(t,x;s)}{y-x}\bigg|<\infty.
\end{equation}
\end{itemize}

\item[\rm(2)]  There is a transition front $u(t,x)$ that is strictly decreasing in space and uniformly Lipschitz continuous in space, that is,
\begin{equation*}
\sup_{x\neq y\atop t\in\R}\bigg|\frac{u(t,y)-u(t,x)}{y-x}\bigg|<\infty,
\end{equation*}
and  a continuously differentiable function $X:\R\to\R$ satisfying the following properties:
\begin{itemize}
\item[\rm(i)] there holds
\begin{equation*}
X(t;s)\to X(t),\quad u(t,x;s)\to u(t,x)\quad\text{and}\quad u_{t}(t,x;s)\to u_{t}(t,x)
\end{equation*}
locally uniformly in $(t,x)\in\R\times\R$ as $s\to-\infty$ along some subsequence;

\item[\rm(ii)]  $\dot{X}(t)\in[c_{\min},c_{\max}]$ for all $t\in\R$, where $c_{\min}$ and $c_{\max}$ are as in $\rm(1)(ii)$;

\item[\rm(iii)] there hold
\begin{equation*}
\begin{split}
u(t,x)&\geq1-e^{c_{-}(x-X(t)+h_{-})}\quad\text{if}\quad x\leq X(t)-h_{-},\\
u(t,x)&\leq e^{-c_{+}(x-X(t)-h_{+})}\quad\text{if}\quad x\geq X(t)+h_{+}
\end{split}
\end{equation*}
for all $s<0$, $t\geq s$, where $c_{\pm}$ and $h_{\pm}$ are as in $\rm(1)(ii)$.
\end{itemize}
\end{itemize}
\end{prop}

In the present paper, we intend to improve the uniform Lipschitz continuity in space of $u(t,x)$ in Proposition
\ref{prop-property-approximating-sol}(2),  and then, study other important properties of $u(t,x)$ such as uniform steepness and stability. To do so, we further assume

\medskip

\noindent {\bf (H3)} {\it  $f(t,u)$ is twice continuously differentiable in $u$ and satisfies
\begin{equation*}
\sup_{(t,u)\in\R\times[0,1]}|f_{uu}(t,u)|<\infty.
\end{equation*}
}

\medskip

Our first main result concerning space regularity of $u(t,x)$ is stated in the following theorem.

\begin{thm}\label{thm-regularity}
Suppose $\rm(H1)$-$\rm(H3)$. Let $u(t,x)$ be the transition front in Proposition \ref{prop-property-approximating-sol}$\rm(2)$. Then, for any $t\in\R$, $u(t,x)$ is continuously differentiable in $x$. Moreover, $u_{x}(t,x)$ is uniformly bounded and uniformly Lipschitz continuous in $x$, that is,
\begin{equation}\label{uniform-bd-space-derivative}
\sup_{(t,x)\in\R\times\R}|u_{x}(t,x)|<\infty\quad\text{and}\quad\sup_{x\neq y\atop t\in\R}\bigg|\frac{u_{x}(t,x)-u_{x}(t,y)}{x-y}\bigg|<\infty,
\end{equation}
respectively.
\end{thm}

We remark that since $u(t,x)$ is strictly decreasing in $x$, the uniform bound of $u_{x}(t,x)$ in \eqref{uniform-bd-space-derivative} is equivalent to $\inf_{(t,x)\in\R\times\R}u_{x}(t,x)>-\infty$. With the regularity, the profile function $\phi(t,x)=u(t,x+X(t))$ satisfies the following evolution equation
\begin{equation*}
\phi_{t}=J\ast\phi-\phi+\dot{X}(t)\phi_{x}+f(t,\phi),
\end{equation*}
which could be used to construct transition fronts if $X(t)$ can be first constructed (see \cite{NaRo14} for the work on \eqref{eqn-classical} in time heterogeneous monostable media).

Next, we study the uniform steepness of the transition front. We prove

\begin{thm}\label{thm-steepness}
Suppose $\rm(H1)$-$\rm(H3)$. Let $u(t,x)$ and $X(t)$ be as in Proposition\ref{prop-property-approximating-sol}$\rm(2)$. Then, for any $M>0$, there holds
\begin{equation*}
\sup_{t\in\R}\sup_{x\in[X(t)-M,X(t)+M]}u_{x}(t,x)<0.
\end{equation*}
\end{thm}

A simple consequence of Theorem \ref{thm-regularity} and Theorem \ref{thm-steepness} is that the interface location at any constant value between $0$ and $1$ is continuously differentiable with finite speed (see Corollary \ref{cor-interface-at-const-value}).

Finally, we study the stability of transition fronts. Let $C_{\rm unif}^{b}(\R,\R)$ be the space of bounded and uniformly continuous functions on $\R$. For $u_{0}\in C_{\rm unif}^{b}(\R,\R)$, denote by $u(t,x;t_{0},u_{0})$ the unique solution of \eqref{main-eqn} with initial data $u(t_{0},\cdot;t_{0},u_{0})=u_{0}$. To state the result, we enhance the last assumption in $\rm(H2)$ and assume

\medskip

\noindent {\bf (H4)} {\it  there exist $\tilde{\theta}\in(\theta,1)$ and $\tilde{\beta}>0$ such that $f_{u}(t,u)\leq-\tilde{\beta}$ for all $(t,u)\in\R\times[\tilde{\theta},2]$.  Also, $f(t,u)=0$ for $(t,u)\in\R\times(-\infty,0)$.}

\medskip

Let $M_{1}>0$ be such that for any $t\in\R$
\begin{equation}\label{condition-1}
u(t,x)\geq\frac{1+\tilde{\theta}}{2}\,\,\text{if}\,\,x-X(t)\leq-M_{1}\quad\text{and}\quad u(t,x)\leq\frac{\theta}{2}\,\,\text{if}\,\,x-X(t)\geq M_{1},
\end{equation}
where $\tilde{\theta}$ is as in $\rm(H4)$. Such an $M_{1}$ exists by Proposition\ref{prop-property-approximating-sol}$\rm(2)(iii)$.
For given $\alpha>0$, let $\Ga_{\al}:\R\to[0,1]$ be a smooth nonincreasing function satisfying
\begin{equation}\label{parameter-1}
\Ga_\alpha(x)=\begin{cases}
1, & x\leq -M_{1}-1,\\
e^{-\al(x-M_{1})},& x\geq M_{1}+1.
\end{cases}
\end{equation}
This function is introduced for making up the lack of asymptotic stability of the equilibrium $0$ (see e.g. \cite{MNRR09, ShSh14-1}). We prove

\begin{thm}\label{thm-stability}
Suppose $\rm(H1)$-$\rm(H4)$.
\begin{itemize}
\item[\rm(1)] There is $\alpha_0>0$ such that for any $0<\alpha\le\alpha_0$, there are $\epsilon_0=\epsilon_0(\alpha)$, $\omega=\omega(\alpha)$, and $A=A(\alpha)$  satisfying that for  any $u_{0}:\R\ra[0,1]$, $u_0\in C_{\rm unif}^{b}(\R,\R)$,
if  there exist $t_{0}\in\R$,  $\ep\in(0,\ep_{0}]$,   $\zeta_{0}^{\pm}$ such that
\begin{equation}\label{initial-condition}
\begin{split}
&u(t_{0},x-\zeta_{0}^{-})-\ep\Ga_\alpha(x-\zeta_{0}^{-}-X(t_{0}))\\
&\quad\quad\leq u_{0}(x)\leq u(t_{0},x-\zeta_{0}^{+})+\ep\Ga_\alpha(x-\zeta_{0}^{+}-X(t_{0}))
\end{split}
\end{equation}
for all $x\in\R$, then, there holds
\begin{equation*}
\begin{split}
&u(t,x-\zeta^{-}(t))-q(t)\Ga_\alpha(x-\zeta^{-}(t)-X(t))\\
&\quad\quad\leq u(t,x;t_{0},u_{0})\leq u(t,x-\zeta^{+}(t))+q(t)\Ga_\alpha(x-\zeta^{+}(t)-X(t))
\end{split}
\end{equation*}
for all $x\in\R$ and $t\geq t_{0}$, where
\begin{equation*}
\zeta^{\pm}(t)=\zeta_{0}^{\pm}\pm\frac{A\ep}{\om}(1-e^{-\om(t-t_{0})})\quad\text{and}\quad q(t)=\ep e^{-\om(t-t_{0})}.
\end{equation*}

\item[\rm(2)] Let $u(t,x)$ and $X(t)$ be as in Proposition\ref{prop-property-approximating-sol}$\rm(2)$. Let $\beta_{0}>0$. Suppose $t_{0}\in\R$ and $u_{0}\in C_{\rm unif}^{b}(\R,\R)$ satisfy
\begin{equation*}
\begin{cases}
u_{0}:\R\to[0,1],\quad u_{0}(-\infty)=1;\\
\exists C>0\,\,\text{s.t.}\,\,|u_{0}-u(t_{0},x)|\leq Ce^{-\beta_{0}(x-X(t_{0}))}\,\,\text{for}\,\,x\in\R.
\end{cases}
\end{equation*}
Then, there exist $\om>0$ and $\tilde\epsilon_0>0$   such that for any $\ep\in(0,\tilde \ep_{0}]$ there are $\zeta^{\pm}=\zeta^{\pm}(\ep,u_{0})\in\R$ such that
\begin{equation*}
u(t,x-\zeta^{-})-\ep e^{-\om(t-t_{0})}\leq u(t,x;t_{0},u_{0})\leq u(t,x-\zeta^{+})+\ep e^{-\om(t-t_{0})}
\end{equation*}
for all $x\in\R$ and $t\geq t_{0}$.
\end{itemize}
\end{thm}

Based on Theorem \ref{thm-stability} and the ``squeezing technique" (see e.g. \cite{Ch97,MNRR09,Sh99-1,ShSh14-1}), we obtain the asymptotic stability.

\begin{thm}\label{thm-asymptotic-stability}
Suppose $\rm(H1)$-$\rm(H4)$. Let $u(t,x)$ and $X(t)$ be as in Proposition\ref{prop-property-approximating-sol}$\rm(2)$. Let $\beta_{0}>0$. Suppose $t_{0}\in\R$ and $u_{0}\in C_{\rm unif}^{b}(\R,\R)$ satisfy
\begin{equation*}
\begin{cases}
u_{0}:\R\to[0,1],\quad u_{0}(-\infty)=1;\\
\exists C>0\,\,\text{s.t.}\,\,|u_{0}-u(t_{0},x)|\leq Ce^{-\beta_{0}(x-X(t_{0}))}\,\,\text{for}\,\,x\in\R.
\end{cases}
\end{equation*}
Then, there exist $C=C(u_{0})>0$, $\zeta_{*}=\zeta_{*}(u_{0})\in\R$ and $r=r(\beta_{0})>0$ such that
\begin{equation*}
\sup_{x\in\R}|u(t,x;t_{0},u_{0})-u(t,x-\zeta_{*})|\leq Ce^{-r(t-t_{0})}
\end{equation*}
for all $t\geq t_{0}$.
\end{thm}

We point out, allowing the solution to develop into the shape satisfying the condition in Theorem \ref{thm-stability} at a later time, Theorem \ref{thm-stability}$\rm(2)$ and Theorem \ref{thm-asymptotic-stability} are true for more general initial data (see Corollary \ref{cor-stability-12345} and Corollary \ref{cor-asymptotic-stability-12345}).

The rest of the paper is organized as follows. In Section \ref{sec-regularity}, we study the space regularity of $u(t,x)$ and prove Theorem \ref{thm-regularity}. In Section \ref{sec-steepness}, we study the uniform steepness of $u(t,x)$ and prove Theorem \ref{thm-steepness}. In Section \ref{sec-stability}, we study the stability of $u(t,x)$ and prove Theorem \ref{thm-stability}. In Section \ref{sec-asymptotic-stability}, we study the asymptotic stability of $u(t,x)$ and prove Theorem \ref{thm-asymptotic-stability}. We also include an appendix, Appendix \ref{sec-app-cp}, on comparison principles for convenience.


\section{Regularity of transition fronts}\label{sec-regularity}

In this section, we study the regularity of $u(t,x)$ and prove Theorem \ref{thm-regularity}. Throughout this section, we assume $\rm(H1)$-$\rm(H3)$. To prove Theorem \ref{thm-regularity}, we first investigate the space regularity of $u(t,x;s)$. We have

\begin{thm}\label{thm-regularity-approx}
For any $s<0$ and $t\geq s$, $u(t,x;s)$ is continuously differentiable in $x$. Moreover,
\begin{itemize}
\item[\rm(i)]$u_{x}(t,x;s)$ is uniformly bounded, that is,
\begin{equation*}
\sup_{x\neq y\atop s<0,t\geq s}|u_{x}(t,x;s)|<\infty;
\end{equation*}
\item[\rm(ii)] $u_{x}(t,x;s)$ is uniformly Lipschitz continuous in space, that is,
\begin{equation*}
\sup_{x\neq y\atop s<0,t\geq s}\bigg|\frac{u_{x}(t,x;s)-u_{x}(t,y;s)}{x-y}\bigg|<\infty.
\end{equation*}
\end{itemize}
\end{thm}

Assuming Theorem \ref{thm-regularity-approx}, let us prove Theorem \ref{thm-regularity}.

\begin{proof}[Proof of Theorem \ref{thm-regularity}]
It follows from Proposition \ref{prop-property-approximating-sol}$\rm(2)(i)$, Theorem \ref{thm-regularity-approx}, Arzel\`{a}-Ascoli theorem and the diagonal argument. More precisely, besides $u(t,x;s)\to u(t,x)$ and $ u_{t}(t,x;s)\to u_{t}(t,x)$ locally uniformly as in Proposition \ref{prop-property-approximating-sol}$\rm(2)(i)$ we also have
\begin{equation}\label{local-unifrom-convergence-derivative}
u_{x}(t,x;s)\to u_{x}(t,x)\quad\text{locally uniformly in}\,\,(t,x)\in\R\times\R
\end{equation}
as $s\to-\infty$ along some subsequence. The properties of $u(t,x)$ then inherit from that of $u(t,x;s)$.
\end{proof}

In the rest of this section, we prove Theorem \ref{thm-regularity-approx}.

\begin{proof}[Proof of Theorem \ref{thm-regularity-approx}]
$\rm(i)$ Setting
\begin{equation*}
v^{\eta}(t,x;s):=\frac{u(t,x+\eta;s)-u(t,x;s)}{\eta}.
\end{equation*}
By \eqref{locally-uniform-Lip}, $\sup_{x\in\R,\eta\neq0\atop s<0,t\geq s}|v^{\eta}(t,x;s)|<\infty$. Clearly, $v^{\eta}(t,x;s)$ satisfies
\begin{equation}\label{an-equation-1}
v^{\eta}_{t}(t,x;s)=\int_{\R}J(x-y)v^{\eta}(t,y;s)dy-v^{\eta}(t,x;s)+a^{\eta}(t,x;s)v^{\eta}(t,x;s),
\end{equation}
where
\begin{equation*}
a^{\eta}(t,x;s)=\frac{f(t,u(t,x+\eta;s))-f(t,u(t,x;s))}{u(t,x+\eta;s)-u(t,x;s)}
\end{equation*}
is uniformly bounded by $\rm(H2)$. Setting
\begin{equation*}
b^{\eta}(t,x;s):=\int_{\R}J(x-y)v^{\eta}(t,y;s)dy=\int_{\R}\frac{J(x-y+\eta)-J(x-y)}{\eta}u(t,y;s)dy,
\end{equation*}
we see that $\sup_{x\in\R,\eta\neq0\atop s<0,t\geq s}|b^{\eta}(t,x;s)|<\infty$, since $J'\in L^{1}(\R)$ and $u(t,x;s)\in(0,1)$.

The solution of \eqref{an-equation-1} is given by
\begin{equation}\label{ODIs-sol}
v^{\eta}(t,x;s)=v^{\eta}(s,x;s)e^{-\int_{s}^{t}(1-a^{\eta}(\tau,x;s))d\tau}+\int_{s}^{t}b^{\eta}(r,x;s)e^{-\int_{r}^{t}(1-a^{\eta}(\tau,x;s))d\tau}dr.
\end{equation}
Notice as $\eta\to0$, the following pointwise limits hold:
\begin{equation*}
\begin{split}
v^{\eta}(s,x;s)&=\frac{\phi_{\min}(x+\eta-y_{s})-\phi_{\min}(x-y_{s})}{\eta}\to\phi_{\min}'(x-y_{s}),\\
a^{\eta}(t,x;s)&\to f_{u}(t,u(t,x;s))\quad\text{and}\\
b^{\eta}(t,x;s)&\to\int_{\R}J'(x-y)u(t,y;s)dy,
\end{split}
\end{equation*}
where $\phi_{\min}$ is as in \eqref{tw-homo}.
Then, setting $\eta\to0$ in \eqref{ODIs-sol}, we conclude from the dominated convergence theorem that for any $s<0$, $t\geq s$ and $x\in\R$, the limit $u_{x}(t,x;s)=\lim_{\eta\to0}v^{\eta}(t,x;s)$ exists and
\begin{equation}\label{formula-for-derivative}
\begin{split}
u_{x}(t,x;s)&=\phi'_{\min}(x-y_{s})e^{-\int_{s}^{t}(1-f_{u}(\tau,u(\tau,x;s)))d\tau}+\int_{s}^{t}b(r,x;s)e^{-\int_{r}^{t}(1-f_{u}(\tau,u(\tau,x;s)))d\tau}dr,
\end{split}
\end{equation}
where $b(t,x;s)=\int_{\R}J'(x-y)u(t,y;s)dy=\int_{\R}J'(y)u(t,x-y;s)dy$. In particular, for any $s<0$ and $t\geq s$, $u(t,x;s)$ is continuously differentiable in $x$. The uniform boundedness of $u_{x}(t,x;s)$, i.e., $\sup_{x\neq y\atop s<0,t\geq s}|u_{x}(t,x;s)|<\infty$, then follows from \eqref{locally-uniform-Lip}.

$\rm(ii)$ Since $u_{x}(t,x;s)$ is uniformly bounded by $\rm(i)$, we trivially have
\begin{equation*}
\forall\de>0,\quad\sup_{|x-y|\geq\de\atop s<0,t\geq s}\bigg|\frac{u_{x}(t,x;s)-u_{x}(t,y;s)}{x-y}\bigg|<\infty.
\end{equation*}
Thus, to show the uniform Lipschitz continuity of $u_{x}(t,x;s)$, it suffices to show the local uniform Lipschitz continuity, i.e.,
\begin{equation}\label{local-uniform-Lip-derivative}
\forall\de>0,\quad\sup_{|x-y|\leq\de\atop s<0,t\geq s}\bigg|\frac{u_{x}(t,x;s)-u_{x}(t,y;s)}{x-y}\bigg|<\infty.
\end{equation}

To this end, we fix $\de>0$. Let $X(t;s)$ and $X_{\la}(t;s)$ for $\la\in(0,1)$ be as in Proposition \ref{prop-property-approximating-sol}$\rm(1)(ii)$ and define
\begin{equation*}
L_{1}=\de+\sup_{s<0,t\geq s}\big|X_{\theta}(t;s)-X(t;s)\big|\quad\text{and}\quad
L_{2}=\de+\sup_{s<0,t\geq s}\big|X_{\tilde{\theta}}(t;s)-X(t;s)\big|,
\end{equation*}
where $\tilde{\theta}\in(\theta,1)$ is given in $\rm(H2)$. Notice $L_{1}<\infty$ and $L_{2}<\infty$ by the uniform exponential decaying estimates in Proposition \ref{prop-property-approximating-sol}$\rm(1)(ii)$. Then, for any $x\in\R$ and $|\eta|\leq\de$ we have
\begin{itemize}
\item if $x\geq X(t;s)+L_{1}$, then $x+\eta\geq x-\de\geq X_{\theta}(t;s)$, which implies that $u(t,x+\eta;s)\leq\theta$ by monotonicity,
and hence
\begin{equation}\label{estimate-aux-1}
f_{u}(t,u(t,x+\eta;s))=0;
\end{equation}

\item if $x\leq X(t;s)-L_{2}$, then $x+\eta\leq x+\de\leq X_{\tilde{\theta}}(t;s)$, which implies that $u(t,x+\eta;s)\geq\tilde{\theta}$ by monotonicity, and hence by $\rm(H2)$,
\begin{equation}\label{estimate-aux-2}
f_{u}(t,u(t,x+\eta;s))\leq0.
\end{equation}
\end{itemize}
According to \eqref{estimate-aux-1} and \eqref{estimate-aux-2}, we consider time-dependent disjoint decompositions of $\R$ into
\begin{equation*}
\R=R_{l}(t;s)\cup R_{m}(t;s)\cup R_{r}(t;s),
\end{equation*}
where
\begin{equation}\label{moving-regions}
\begin{split}
R_{l}(t;s)&=(-\infty,X(t;s)-L_{2}),\\
R_{m}(t;s)&=[X(t;s)-L_{2}, X(t;s)+L_{1}]\quad\text{and}\\
R_{r}(t;s)&=(X(t;s)+L_{1},\infty).
\end{split}
\end{equation}

For $s<0$ and $x_{0}\in\R$, let $t_{\rm first}(x_{0};s)$ be the first time that $x_{0}$ is in $R_{m}(t;s)$, that is,
\begin{equation*}
t_{\rm first}(x_{0};s)=\min\big\{t\geq s\big|x_{0}\in R_{m}(t;s)\big\},
\end{equation*}
and $t_{\rm last}(x_{0};s)$ be the last time that $x_{0}$ is in $R_{m}(t;s)$, that is,
\begin{equation*}
t_{\rm last}(x_{0};s)=\max\big\{t_{0}\in\R\big|x_{0}\in R_{m}(t_{0};s)\,\,\text{and}\,\,x_{0}\notin R_{m}(t,s)\,\,\text{for}\,\,t>t_{0}\big\}.
\end{equation*}
Since $\dot{X}(t;s)\geq c_{\min}>0$ by Proposition \ref{prop-property-approximating-sol}$\rm(1)(ii)$, if $x_{0}\in R_{l}(s;s)$, then $x_{0}\in R_{l}(t;s)$ for all $t>s$. In this case, $t_{\rm first}(x_{0};s)$ and $t_{\rm last}(x_{0};s)$ are not well-defined, but it will not cause any trouble. We see that $t_{\rm first}(x_{0};s)$ and $t_{\rm last}(x_{0};s)$ are well-defined only if $x_{0}\notin R_{l}(s;s)$. As a simple consequence of $\dot{X}(t;s)\in[c_{\min},c_{\max}]$ in Proposition \ref{prop-property-approximating-sol}$\rm(1)(ii)$ and the fact that the length of $R_{m}(t;s)$ is $L_{1}+L_{2}$, we have
\begin{equation}\label{growth-period}
T=T(\de):=\sup_{s<0,x_{0}\notin R_{l}(s;s)}\big[t_{\rm last}(x_{0};s)-t_{\rm first}(x_{0};s)\big]<\infty.
\end{equation}
Moreover, we see that for any $|\eta|\leq\de$,
\begin{equation}\label{space-derivative-of-f}
\begin{split}
f_{u}(t,u(t,x_{0}+\eta;s))&=0\quad\text{if}\quad t\in[s,t_{\rm first}(x_{0};s)],\\
f_{u}(t,u(t,x_{0}+\eta;s))&\leq0\quad\text{if}\quad t\geq t_{\rm last}(x_{0};s).
\end{split}
\end{equation}

We now show that
\begin{equation}\label{local-uniform-Lip-derivative-de}
\sup_{x_{0}\in\R,0<|\eta|\leq\de\atop s<0,t\geq s}\bigg|\frac{u_{x}(t,x_{0}+\eta;s)-u_{x}(t,x_{0};s)}{\eta}\bigg|<\infty.
\end{equation}
Using \eqref{formula-for-derivative}, we have
\begin{equation*}
\begin{split}
&\frac{u_{x}(t,x_{0}+\eta;s)-u_{x}(t,x_{0};s)}{\eta}\\
&\quad\quad=\underbrace{\frac{\phi_{\min}'(x_{0}+\eta-y_{s})-\phi_{\min}'(x_{0}-y_{s})}{\eta}e^{-\int_{s}^{t}(1-f_{u}(\tau,u(\tau,x_{0}+\eta;s)))d\tau}}_{\rm (I)}\\
&\quad\quad\quad+\underbrace{\phi_{\min}'(x_{0}-y_{s})\frac{e^{-\int_{s}^{t}(1-f_{u}(\tau,u(\tau,x_{0}+\eta;s)))d\tau}-e^{-\int_{s}^{t}(1-f_{u}(\tau,u(\tau,x_{0};s)))d\tau}}{\eta}}_{\rm (II)}\\
&\quad\quad\quad+\underbrace{\int_{s}^{t}\frac{b(r,x_{0}+\eta;s)-b(r,x_{0};s)}{\eta}e^{-\int_{r}^{t}(1-f_{u}(\tau,u(\tau,x_{0};s)))d\tau}dr}_{\rm (III)}\\
&\quad\quad\quad+\underbrace{\int_{s}^{t}b(r,x_{0};s)\frac{e^{-\int_{r}^{t}(1-f_{u}(\tau,u(\tau,x_{0}+\eta;s)))d\tau}-e^{-\int_{r}^{t}(1-f_{u}(\tau,u(\tau,x_{0};s)))d\tau}}{\eta}dr}_{\rm (IV)}.
\end{split}
\end{equation*}
Hence, it suffice to bound terms $\rm (I)$-$\rm (IV)$. To do so, we need to consider three cases: $x_{0}\in R_{l}(s;s)$, $x_{0}\in R_{m}(s;s)$ and $x_{0}\in R_{r}(s;s)$. We here focus on the last case, i.e., $x_{0}\in R_{r}(s;s)$, which is the most involved one. The other two cases are simpler and can be treated similarly. Also, for fixed $s<0$ and $x_{0}\in R_{r}(s;s)$, we will focus on $t\geq t_{\rm last}(x_{0};s)$; the case with $t\in[t_{\rm first}(x_{0};s),t_{\rm last}(x_{0};s)]$ or $t\leq t_{\rm first}(x_{0};s)$ will be clear. Thus, we assume $x_{0}\in R_{r}(s;s)$ and $t\geq t_{\rm last}(x_{0};s)$.

We will frequently use the following estimates: for any $|\tilde{\eta}|\leq\de$ there hold
\begin{equation}\label{some-estimate}
\begin{split}
e^{-\int_{r}^{t_{\rm first}(x_{0};s)}(1-f_{u}(\tau,u(\tau,x_{0}+\tilde{\eta};s)))d\tau}&=e^{-(t_{\rm first}(x_{0};s)-r)},\quad r\in[s,t_{\rm first}(x_{0};s)]\\
e^{-\int_{r}^{t_{\rm last}(x_{0};s)}(1-f_{u}(\tau,u(\tau,x_{0}+\tilde{\eta};s)))d\tau}&\leq e^{T\sup_{(t,u)\in\R\times[0,1]}|1-f_{u}(t,u)|},\quad r\in[t_{\rm first}(x_{0};s),t_{\rm last}(x_{0};s)]\\
e^{-\int_{r}^{t}(1-f_{u}(\tau,u(\tau,x_{0}+\tilde{\eta};s)))d\tau}&\leq e^{-(t-r)},\quad r\in[t_{\rm last}(x_{0};s),t].
\end{split}
\end{equation}
They are simple consequences of \eqref{growth-period} and \eqref{space-derivative-of-f}. Set
\begin{equation*}
\begin{split}
&C_{0}:=\sup_{(t,u)\in\R\times[0,1]}|1-f_{u}(t,u)|,\quad C_{1}:=\sup_{x\neq y}\bigg|\frac{\phi_{\min}'(x)-\phi_{\min}'(y)}{x-y}\bigg|, \quad C_{2}:=\sup_{x\in\R}|\phi_{\min}'(x)|\\
&C_{3}:=\sup_{(t,u)\in\R\times[0,1]}|f_{uu}(t,u)|\times\sup_{x\in\R\atop s<0,t\geq s}|u_{x}(t,x;s)|,\quad C_{4}=\sup_{x\neq y\atop s<0,t\geq s}\bigg|\frac{u(t,x;s)-u(t,y;s)}{x-y}\bigg|.
\end{split}
\end{equation*}
Note that all these constants are finite. In fact, $C_{0}<\infty$ by $\rm(H2)$, $C_{1}<\infty$ by \eqref{Lip-TW}, $C_{3}<\infty$ by $\rm(H3)$ and Theorem \ref{thm-regularity-approx}$\rm(i)$, and $C_{4}<\infty$ by Proposition \ref{prop-property-approximating-sol}$\rm(1)(iii)$.

We are ready to bound $\rm(I)$-$\rm(IV)$. For the term $\rm (I)$, using \eqref{Lip-TW} and \eqref{some-estimate}, we see that
\begin{equation}\label{term-I}
\begin{split}
|{\rm (I)}|&\leq C_{1}e^{-\int_{s}^{t}(1-f_{u}(\tau,u(\tau,x_{0}+\eta;s)))d\tau}\\
&=C_{1}e^{-\big[\int_{s}^{t_{\rm first}(x_{0};s)}+\int_{t_{\rm first}(x_{0};s)}^{t_{\rm last}(x_{0};s)}+\int_{t_{\rm last}(x_{0};s)}^{t}\big](1-f_{u}(\tau,u(\tau,x_{0}+\eta;s)))d\tau}\\
&\leq C_{1}e^{-(t_{\rm first}(x_{0};s)-s)}e^{C_{0}T}e^{-(t-t_{\rm last}(x_{0};s))}\leq C_{1}e^{C_{0}T}.
\end{split}
\end{equation}

For the term $\rm (II)$, we have from Taylor expansion of the function $\eta\mapsto e^{-\int_{s}^{t}(1-f_{u}(\tau,u(\tau,x_{0}+\eta;s)))d\tau}$ at $\eta=0$ that
\begin{equation*}
\begin{split}
|{\rm (II)}|&\leq C_{2}\bigg|\frac{e^{-\int_{s}^{t}(1-f_{u}(\tau,u(\tau,x_{0}+\eta;s)))d\tau}-e^{-\int_{s}^{t}(1-f_{u}(\tau,u(\tau,x_{0};s)))d\tau}}{\eta}\bigg|\\
&\leq C_{2}e^{-\int_{s}^{t}(1-f_{u}(\tau,u(\tau,x_{0}+\eta_{*};s)))d\tau}\int_{s}^{t}\Big|f_{uu}(\tau,u(\tau,x_{0}+\eta_{*};s))u_{x}(\tau,x_{0}+\eta_{*};s)\Big|d\tau,
\end{split}
\end{equation*}
where $\eta_{*}$ is between $0$ and $\eta$, and hence, $|\eta_{*}|\leq \de$. We see
\begin{equation*}
\int_{s}^{t}\Big|f_{uu}(\tau,u(\tau,x_{0}+\eta_{*};s))u_{x}(\tau,x_{0}+\eta_{*};s)\Big|d\tau\leq C_{3}(t-s).
\end{equation*}
It then follows from \eqref{some-estimate} that
\begin{equation}\label{term-II}
\begin{split}
|{\rm (II)}|&\leq C_{2}C_{3}e^{-(t_{\rm first}(x_{0};s)-s)}e^{-(t-t_{\rm last}(x_{0};s))}(t-s)\\
&=C_{2}C_{3} e^{-(t_{\rm first}(x_{0};s)-s)}e^{-(t-t_{\rm last}(x_{0};s))}\\
&\quad\times\big[(t-t_{\rm last}(x_{0};s))+(t_{\rm last}(x_{0};s)-t_{\rm first}(x_{0};s))+(t_{\rm first}(x_{0};s)-s)\big]\\
&\leq C_{2}C_{3}\big[e^{-(t-t_{\rm last}(x_{0};s))}(t-t_{\rm last}(x_{0};s))+T+ e^{-(t_{\rm first}(x_{0};s)-s)}(t_{\rm first}(x_{0};s)-s)\big]\\
&\leq C_{2}C_{3}\bigg(\frac{2}{e}+T\bigg).
\end{split}
\end{equation}

For the term $\rm(III)$, we first see that
\begin{equation*}
\bigg|\frac{b(r,x_{0}+\eta;s)-b(r,x_{0};s)}{\eta}\bigg|=\bigg|\int_{\R}J'(y)\frac{u(r,x_{0}+\eta-y;s)-u(r,x_{0}-y;s)}{\eta}dy\bigg|\leq C_{4}\|J'\|_{L^{1}(\R)}.
\end{equation*}
Thus,
\begin{equation*}
\begin{split}
|{\rm \rm(III)}|&\leq C_{4}\|J'\|_{L^{1}(\R)}\int_{s}^{t}e^{-\int_{r}^{t}(1-f_{u}(\tau,u(\tau,x_{0};s)))d\tau}dr\\
&=C_{4}\|J'\|_{L^{1}(\R)}\bigg[\underbrace{\int_{s}^{t_{\rm first}(x_{0};s)}e^{-\int_{r}^{t}(1-f_{u}(\tau,u(\tau,x_{0};s)))d\tau}dr}_{\rm(III\mbox{-}1)}+\underbrace{\int_{t_{\rm first}(x_{0};s)}^{t_{\rm last}(x_{0};s)}e^{-\int_{r}^{t}(1-f_{u}(\tau,u(\tau,x_{0};s)))d\tau}dr}_{\rm(III\mbox{-}2)}\\
&\quad\quad\quad\quad\quad\quad\quad+\underbrace{\int_{t_{\rm last}(x_{0};s)}^{t}e^{-\int_{r}^{t}(1-f_{u}(\tau,u(\tau,x_{0};s)))d\tau}dr}_{\rm(III\mbox{-}3)}\bigg].
\end{split}
\end{equation*}
We estimate $\rm(III\mbox{-}1)$, $\rm(III\mbox{-}2)$ and $\rm(III\mbox{-}3)$. For $\rm(III\mbox{-}1)$, we obtain from \eqref{some-estimate} that
\begin{equation*}
\begin{split}
{\rm(III\mbox{-}1)}&=\int_{s}^{t_{\rm first}(x_{0};s)}e^{-\big[\int_{r}^{t_{\rm first}(x_{0};s)}+\int_{t_{\rm first}(x_{0};s)}^{t_{\rm last}(x_{0};s)}+\int_{t_{\rm last}(x_{0};s)}^{t}\big](1-f_{u}(\tau,u(\tau,x_{0};s)))d\tau}dr\\
&\leq e^{C_{0}T}\int_{s}^{t_{\rm first}(x_{0};s)}e^{-(t_{\rm first}(x_{0};s)-r)}e^{-(t-t_{\rm last}(x_{0};s))}dr\\
&=e^{C_{0}T}e^{-(t-t_{\rm last}(x_{0};s))}(1-e^{-(t_{\rm first}(x_{0};s)-s)})\leq e^{C_{0}T}.
\end{split}
\end{equation*}
Similarly,
\begin{equation*}
\begin{split}
{\rm(III\mbox{-}2)}&=\int_{t_{\rm first}(x_{0};s)}^{t_{\rm last}(x_{0};s)}e^{-\big[\int_{r}^{t_{\rm last}(x_{0};s)}+\int_{t_{\rm last}(x_{0};s)}^{t}\big](1-f_{u}(\tau,u(\tau,x_{0};s)))d\tau}dr\\
&\leq e^{C_{0}T}\int_{t_{\rm first}(x_{0};s)}^{t_{\rm last}(x_{0};s)}e^{-(t-t_{\rm last}(x_{0};s))}dr\leq e^{C_{0}T}Te^{-(t-t_{\rm last}(x_{0};s))}\leq Te^{C_{0}T}
\end{split}
\end{equation*}
and ${\rm(III\mbox{-}3)}\leq\int_{t_{\rm last}(x_{0};s)}^{t}e^{-(t-r)}dr=1-e^{-(t-t_{\rm last}(x_{0};s))}\leq1$. Hence,
\begin{equation}\label{term-III}
{\rm(III)}\leq C_{4}\|J'\|_{L^{1}(\R)}(e^{C_{0}T}+Te^{C_{0}T}+1).
\end{equation}

For the term $\rm(IV)$, using $|b(r,x_{0};s)|\leq\|J'\|_{L^{1}(\R)}$ and Taylor expansion as in the treatment of the term $\rm(II)$, we have
\begin{equation*}
\begin{split}
|{\rm (IV)}|&\leq\|J'\|_{L^{1}(\R)}\int_{s}^{t}e^{-\int_{r}^{t}(1-f_{u}(\tau,u(\tau,x_{0}+\eta_{*};s)))d\tau}\bigg(\int_{r}^{t}\Big|f_{uu}(\tau,u(\tau,x_{0}+\eta_{*};s))u_{x}(\tau,x_{0}+\eta_{*};s)\Big|d\tau\bigg)dr\\
&\leq C_{3}\|J'\|_{L^{1}(\R)}\int_{s}^{t}(t-r)e^{-\int_{r}^{t}(1-f_{u}(\tau,u(\tau,x_{0}+\eta_{*};s)))d\tau}dr\\
&=C_{3}\|J'\|_{L^{1}(\R)}\bigg[\underbrace{\int_{s}^{t_{\rm first}(x_{0};s)}(t-r)e^{-\int_{r}^{t}(1-f_{u}(\tau,u(\tau,x_{0}+\eta_{*};s)))d\tau}dr}_{\rm(IV\mbox{-}1)}\\
&\quad\quad\quad\quad\quad\quad\quad+\underbrace{\int_{t_{\rm first}(x_{0};s)}^{t_{\rm last}(x_{0};s)}(t-r)e^{-\int_{r}^{t}(1-f_{u}(\tau,u(\tau,x_{0}+\eta_{*};s)))d\tau}dr}_{\rm(IV\mbox{-}2)}\\
&\quad\quad\quad\quad\quad\quad\quad+\underbrace{\int_{t_{\rm last}(x_{0};s)}^{t}(t-r)e^{-\int_{r}^{t}(1-f_{u}(\tau,u(\tau,x_{0}+\eta_{*};s)))d\tau}dr}_{\rm(IV\mbox{-}3)}\bigg],
\end{split}
\end{equation*}
where $|\eta_{*}|\leq|\eta|\leq\de$. Similar to $\rm(III\mbox{-}1)$, $\rm(III\mbox{-}2)$ and $\rm(III\mbox{-}3)$, we have
\begin{equation*}
\begin{split}
{\rm(IV\mbox{-}1)}&=\int_{s}^{t_{\rm first}(x_{0};s)}(t-r)e^{-\big[\int_{r}^{t_{\rm first}(x_{0};s)}+\int_{t_{\rm first}(x_{0};s)}^{t_{\rm last}(x_{0};s)}+\int_{t_{\rm last}(x_{0};s)}^{t}\big](1-f_{u}(\tau,u(\tau,x_{0}+\eta_{*};s)))d\tau}dr\\
&\leq e^{C_{0}T}\int_{s}^{t_{\rm first}(x_{0};s)}\big[(t-t_{\rm last}(x_{0};s))+T+(t_{\rm first}(x_{0};s)-r)\big]e^{-(t_{\rm first}(x_{0};s)-r)}e^{-(t-t_{\rm last}(x_{0};s))}dr\\
&\leq e^{C_{0}T}\bigg[(t-t_{\rm last}(x_{0};s))e^{-(t-t_{\rm last}(x_{0};s))}\int_{s}^{t_{\rm first}(x_{0};s)}e^{-(t_{\rm first}(x_{0};s)-r)}dr\\
&\quad\quad\quad+T\int_{s}^{t_{\rm first}(x_{0};s)}e^{-(t_{\rm first}(x_{0};s)-r)}dr+\int_{s}^{t_{\rm first}(x_{0};s)}(t_{\rm first}(x_{0};s)-r))e^{-(t_{\rm first}(x_{0};s)-r)}dr\bigg]\\
&\leq e^{C_{0}T}\bigg[\frac{1-e^{-(t_{\rm first}(x_{0};s)-s)}}{e}+T(1-e^{-(t_{\rm first}(x_{0};s)-s)})+\Big(1-(1+t_{\rm first}(x_{0};s)-s)e^{-(t_{\rm first}(x_{0};s)-s)}\Big)\bigg]\\
&\leq e^{C_{0}T}\bigg(\frac{1}{e}+T+1\bigg),
\end{split}
\end{equation*}
\begin{equation*}
\begin{split}
{\rm(IV\mbox{-}2)}&=\int_{t_{\rm first}(x_{0};s)}^{t_{\rm last}(x_{0};s)}(t-r)e^{-\big[\int_{r}^{t_{\rm last}(x_{0};s)}+\int_{t_{\rm last}(x_{0};s)}^{t}\big](1-f_{u}(\tau,u(\tau,x_{0}+\eta_{*};s)))d\tau}dr\\
&\leq e^{C_{0}T}\int_{t_{\rm first}(x_{0};s)}^{t_{\rm last}(x_{0};s)}[(t-t_{\rm last}(x_{0};s))+(t_{\rm last}(x_{0};s)-r)]e^{-(t-t_{\rm last}(x_{0};s))}dr\\
&\leq e^{C_{0}T}\bigg[T(t-t_{\rm last}(x_{0};s))e^{-(t-t_{\rm last}(x_{0};s))}+\int_{t_{\rm first}(x_{0};s)}^{t_{\rm last}(x_{0};s)}(t_{\rm last}(x_{0};s)-r)dr\bigg]\\
&\leq e^{C_{0}T}\bigg(\frac{T}{e}+\frac{T^{2}}{2}\bigg)
\end{split}
\end{equation*}
and
\begin{equation*}
\begin{split}
{\rm(IV\mbox{-}3)}&\leq\int_{t_{\rm last}(x_{0};s)}^{t}(t-r)e^{-(t-r)}dr=1-(1+t-t_{\rm last}(x_{0};s))e^{-(t-t_{\rm last}(x_{0};s))}\leq1.
\end{split}
\end{equation*}
Hence,
\begin{equation}\label{term-IV}
|{\rm(IV)}|\leq C_{3}\|J'\|_{L^{1}(\R)}\bigg[e^{C_{0}T}\bigg(\frac{1}{e}+T+1\bigg)+e^{C_{0}T}\bigg(\frac{T}{e}+\frac{T^{2}}{2}\bigg)+1\bigg].
\end{equation}

Consequently, \eqref{local-uniform-Lip-derivative-de} follows from \eqref{term-I}, \eqref{term-II}, \eqref{term-III} and \eqref{term-IV}.
\end{proof}


\section{Uniform steepness}\label{sec-steepness}

In this section, we study the steepness of transition fronts and prove Theorem \ref{thm-steepness}. Throughout this section,
we  assume $\rm(H1)$-$\rm(H3)$. Theorem \ref{thm-steepness} will be a simple result of the following theorem.

\begin{thm}\label{thm-uniform-steepness-approx}
For any $M>0$, there exists $\al_{M}>0$ such that
\begin{equation*}
\sup_{x\in[X(t;s)-M,X(t;s)+M]}u_{x}(t,x;s)\leq-\al_{M}
\end{equation*}
for all $s<0$, $t\geq s$.
\end{thm}

Assuming Theorem \ref{thm-uniform-steepness-approx}, we prove Theorem \ref{thm-steepness}.

\begin{proof}[Proof of Theorem \ref{thm-steepness}]
It follows from Proposition \ref{prop-property-approximating-sol}$\rm(2)(i)$, \eqref{local-unifrom-convergence-derivative} and Theorem \ref{thm-uniform-steepness-approx}.
\end{proof}

To finish the proof of Theorem \ref{thm-steepness}, we prove Theorem \ref{thm-uniform-steepness-approx}, which is based on the following Lemma, whose proof is inspired by the proof of \cite[Theorem 5.1]{Ch97} and \cite[Lemma 3.2]{Sh99-1}.

\begin{lem}\label{lem-tech-1234567}
For any $t>t_{0}\geq s$, $h>0$ and $z\in\R$, there holds
\begin{equation*}
u_{x}(t,x;s)\leq C\int_{z-h}^{z+h}u_{x}(t_{0},y;s)dy,\quad \forall x\in\R,
\end{equation*}
where $C=C(t-t_{0},|x-z|,h)>0$ satisfies
\begin{itemize}
\item[\rm(i)] $C\to0$ polynomially as $t-t_{0}\to0$ and $C\to0$ exponentially as $t-t_{0}\to\infty$;
\item[\rm(ii)] $C:(0,\infty)\times[0,\infty)\times(0,\infty)\to(0,\infty)$ is locally uniformly positive in the sense that for any $0<t_{1}<t_{2}<\infty$, $M_{1}>0$ and $h_{1}>0$, there holds
\begin{equation*}
\inf_{t\in[t_{1},t_{2}], M\in[0,M_{1}], h\in(0,h_{1}]} C(t,M,h)>0.
\end{equation*}
\end{itemize}
\end{lem}
\begin{proof}
Let $\ep>0$. Let $v_{1}(t,x;s)=u(t,x+\ep;s)$ and $v_{2}(t,x;s)=u(t,x;s)$. We see that $v(t,x;s):=v_{1}(t,x;s)-v_{2}(t,x;s)<0$ by monotonicity and satisfies
\begin{equation*}
v_{t}=J\ast v-v+f(t,v_{1})-f(t,v_{2}).
\end{equation*}
By $\rm(H2)$, we can find $K>0$ such that $f(t,v_{1})-f(t,v_{2})\leq -K(v_{1}-v_{2})$, which implies that
\begin{equation*}
v_{t}\leq J\ast v-v-Kv.
\end{equation*}
Setting $\tilde{v}(t,x;s)=e^{(1+K)(t-t_{0})}v(t,x;s)$, we see
\begin{equation}\label{an-differential-inequality-1}
\tilde{v}_{t}\leq J\ast\tilde{v}.
\end{equation}
Since $v<0$, we have $\tilde{v}<0$, which implies $J\ast\tilde{v}<0$ by the nonnegativity of $J$ by $\rm(H1)$, and therefore, $\tilde{v}_{t}<0$ by \eqref{an-differential-inequality-1}. In particular, $\tilde{v}(t,x;s)<\tilde{v}(t_{0},x;s)$. It then follows from the nonnegativity of $J$ and \eqref{an-differential-inequality-1} that
\begin{equation}\label{an-differential-inequality-2}
\tilde{v}_{t}(t,x;s)\leq [J\ast\tilde{v}(t,\cdot;s)](x)\leq[J\ast\tilde{v}(t_{0},\cdot;s)](x).
\end{equation}
For each $x\in\R$, \eqref{an-differential-inequality-2} is an ordinary differential inequality. Integrating \eqref{an-differential-inequality-1} over $[t_{0},t]$ with respect to the time variable, we find from $\tilde{v}(t_{0},x;s)<0$ that
\begin{equation*}
\tilde{v}(t,x;s)\leq (t-t_{0})[J\ast\tilde{v}(t_{0},\cdot;s)](x)+\tilde{v}(t_{0},x;s)<(t-t_{0})[J\ast\tilde{v}(t_{0},\cdot;s)](x).
\end{equation*}
In particular, for any $T>0$, we have
\begin{equation}\label{initial-estimate}
\tilde{v}(t_{0}+T,x;s)<T[J\ast\tilde{v}(t_{0},\cdot;s)](x).
\end{equation}

Then, considering \eqref{an-differential-inequality-1} with initial time at $t_{0}+T$ and repeating the above arguments, we find
\begin{equation*}
\tilde{v}(t_{0}+T+T,x;s)<T[J\ast\tilde{v}(t_{0}+T,\cdot;s)](x)<T^{2}[J\ast J\ast\tilde{v}(t_{0},\cdot;s)](x),
\end{equation*}
where we used \eqref{initial-estimate} in the second inequality. Repeating this, we conclude that for any $T>0$ and any $N=1,2,3,\dots$, there holds
\begin{equation}\label{repeating-result}
\tilde{v}(t_{0}+NT,x;s)<T^{N}[J^{N}\ast\tilde{v}(t_{0},\cdot;s)](x),
\end{equation}
where $J^{N}=\underbrace{J\ast J\ast\cdots\ast J}_{N\,\,\text{times}}$. Note that $J^{N}$ is nonnegative, and if $J$ is compactly supported, then $J^{N}$ is not everywhere positive no matter how large $N$ is. But, since $J$ is nonnegative and positive on some open interval, $J^{N}$ can be positive on any fixed bounded interval if $N$ is large. Moreover, since $J$ is symmetric, so is $J^{N}$.

Now, let $x\in\R$, $z\in\R$ and $h>0$, and let $N:=N(|x-z|,h)$ be large enough so that
\begin{equation*}
\tilde{C}=\tilde{C}(|x-z|,h):=\inf_{y\in[x-z-h,x-z+h]}J^{N}(y)>0.
\end{equation*}
Note that the dependence of $N$ on $x-z$ through $|x-z|$ is due to the symmetry of $J^{N}$. Moreover, the positivity of $\tilde{C}:[0,\infty)\times(0,\infty)\to(0,\infty)$ is uniform on compacts sets, which is because  $N$ can be chosen to be nondecreasing in $|x-z|$ and in $h$.

Then, for $t>t_{0}$, we see from \eqref{repeating-result} with $T=\frac{t-t_{0}}{N}$ that
\begin{equation*}
\begin{split}
\tilde{v}(t,x;s)&<\bigg(\frac{t-t_{0}}{N}\bigg)^{N}\int_{\R}J^{N}(x-y)\tilde{v}(t_{0},y;s)dy\\
&\leq\bigg(\frac{t-t_{0}}{N}\bigg)^{N}\int_{z-h}^{z+h}J^{N}(x-y)\tilde{v}(t_{0},y;s)dy\\
&\leq\tilde{C}\bigg(\frac{t-t_{0}}{N}\bigg)^{N}\int_{z-h}^{z+h}\tilde{v}(t_{0},y;s)dy,
\end{split}
\end{equation*}
since $x-y\in[x-z-h,x-z+h]$ when $y\in[z-h,z+h]$. Going back to $u(t,x;s)$, we find
\begin{equation*}
u(t,x+\ep;s)-u(t,x;s)\leq\tilde{C}e^{-(1+K)(t-t_{0})}\bigg(\frac{t-t_{0}}{N}\bigg)^{N}\int_{z-h}^{z+h}[u(t_{0},y+\ep;s)-u(t_{0},y;s)]dy
\end{equation*}
Dividing the above estimate by $\ep$, we conclude the result from dominated convergence theorem with $C=\tilde{C}e^{-(1+K)(t-t_{0})}\big(\frac{t-t_{0}}{N}\big)^{N}$. From which, we obtain the properties of $C$ and finish the proof.
\end{proof}

Now, we prove Theorem \ref{thm-uniform-steepness-approx} Recall that $\phi_{\min}$ is as in \eqref{tw-homo}, and $c_{\min}$ and $c_{\max}$ are as in  Proposition \ref{prop-property-approximating-sol}$\rm(1)(ii)$.

\begin{proof}[Proof of Theorem \ref{thm-uniform-steepness-approx}]
Set
\begin{equation*}
h_{\theta}:=\max\bigg\{\sup_{s<0,t\geq s}|X(t;s)-X_{\frac{\theta}{2}}(t;s)|,\sup_{s<0,t\geq s}|X(t;s)-X_{\frac{1+\theta}{2}}(t;s)|\bigg\}.
\end{equation*}
By Proposition \ref{prop-property-approximating-sol}$\rm(1)(ii)$, $h_{\theta}<\infty$. Then,
\begin{equation}\label{simple-comparison-12345}
X(t_{0};s)+h_{\theta}\geq X_{\frac{\theta}{2}}(t_{0};s),\quad X(t_{0};s)-h_{\theta}\leq X_{\frac{1+\theta}{2}}(t_{0};s)
\end{equation}
for all $t_{0}\geq s$. Now, for any $\tau\geq0$ and $t_{0}\geq s$, we apply Lemma \ref{lem-tech-1234567} with $z=X(t_{0};s)$ and $h=h_{\theta}$ to see that if $|x-X(t_{0};s)|\leq M$, then
\begin{equation}\label{aprior-estimate-1234567}
\begin{split}
u_{x}(\tau+t_{0},x;s)&\leq C(\tau,M,h_{\theta})\int_{X(t_{0};s)-h_{\theta}}^{X(t_{0};s)+h_{\theta}}u_{x}(t_{0},y;s)dy\\
&= C(\tau,M,h_{\theta})[u(t_{0},X(t_{0};s)+h_{\theta};s)-u(t_{0},X(t_{0};s)-h_{\theta};s)]\\
&\leq C(\tau,M,h_{\theta})[u(t_{0},X_{\frac{\theta}{2}}(t_{0};s);s)-u(t_{0},X_{\frac{1+\theta}{2}}(t_{0};s);s)]\\
&=-\frac{C(\tau,M,h_{\theta})}{2},
\end{split}
\end{equation}
where we used \eqref{simple-comparison-12345} and the monotonicity in the second inequality. Notice $C(\tau,M,h_{\theta})\to0$ as $\tau\to0$.

To apply \eqref{aprior-estimate-1234567}, we see that if $|x-X(t_{0}+1;s)|\leq M$, then
\begin{equation*}
|x-X(t_{0};s)|\leq|x-X(t_{0}+1;s)|+|X(t_{0}+1;s)-X(t_{0};s)|\leq M+c_{\max},
\end{equation*}
where we used Proposition \ref{prop-property-approximating-sol}$\rm(1)(ii)$. We then apply \eqref{aprior-estimate-1234567} with $M$ replaced by $M+c_{\max}$ and $\tau$ replaced by $1$ to conclude that
\begin{equation*}
u_{x}(t_{0}+1,x;s)\leq-\frac{C(1,M+c_{\max},h_{\theta})}{2}.
\end{equation*}
Since $t_{0}\geq s$ is arbitrary, we have shown
\begin{equation*}
\sup_{s<0, t-s\geq 1}\sup_{x\in[X(t;s)-M,X(t;s)+M]}u_{x}(t,x;s)<0.
\end{equation*}

To finish the proof, we only need to show
\begin{equation}\label{finite-time-steepness}
\sup_{s<0, 0\leq t-s\leq 1}\sup_{x\in[X(t;s)-M,X(t;s)+M]}u_{x}(t,x;s)<0.
\end{equation}
To this end, we recall
\begin{equation}\label{formula-for-derivative-1}
u_{x}(t,x;s)=\phi'_{\min}(x-y_{s})e^{-\int_{s}^{t}(1-f_{u}(\tau,u(\tau,x;s)))d\tau}+\int_{s}^{t}b(r,x;s)e^{-\int_{r}^{t}(1-f_{u}(\tau,u(\tau,x;s)))d\tau}dr,
\end{equation}
where $b(t,x;s)=\int_{\R}J'(x-y)u(t,y;s)dy=\int_{\R}J(x-y)u_{x}(t,y;s)dy$. It is just the solution of the initial-value problem
\begin{equation*}
(u_{x})_{t}=J\ast u_{x}-u_{x}+f_{u}(t,u)u_{x},\quad u_{x}(s,x;s)=\phi_{\min}'(x-y_{s}).
\end{equation*}
Set $a:=\inf_{(t,x)\in\R\times[0,1]}f_{u}(t,u)<0$. Since $\phi_{\min}'<0$ and $b(t,x;s)<0$, \eqref{formula-for-derivative-1} implies
\begin{equation*}
u_{x}(t,x;s)\leq\phi'_{\min}(x-y_{s})e^{-(1-a)(t-s)}+\int_{s}^{t}b(r,x;s)e^{-(1-a)(t-r)}dr.
\end{equation*}
In particular,
\begin{equation}\label{formula-for-derivative-2}
u_{x}(t,x;s)\leq
\phi'_{\min}(x-y_{s})e^{-(1-a)}\quad\text{if}\quad0\leq t-s\leq 1
\end{equation}

For $0\leq t-s\leq 1$, we have from $\dot{X}(t;s)\in[c_{\min},c_{\max}]$ by Proposition \ref{prop-property-approximating-sol}$\rm(1)(ii)$ that
\begin{equation}\label{an-estimate-1234567}
X(t;s)-X(s;s)\in[c_{\min}(t-s),c_{\max}(t-s)]\subset[0,c_{\max}].
\end{equation}
Recall that for any $\la\in(0,1)$, $X_{\la}(t;s)$ is such that $u(t,X_{\la}(t;s);s)=\la$. In particular, $X_{\la}(s;s)$ is such that $\phi_{\min}(X_{\la}(s;s)-y_{s})=\la$. Thus, $X_{\la}(s;s)-y_{s}$ is independent of $s$. From the construction of $X(t;s)$ in Theorem \cite[Theorem 4.1]{ShSh14-2}, we know $X(s;s)=X_{\la_{*}}(s;s)+C_{1}$ for some $\la_{*}\in(\theta,1)$ and $C_{1}>0$. Hence, there exists $C_{2}\in\R$ such that $X(s;s)=y_{s}+C_{2}$ for all $s<0$, which, together with \eqref{an-estimate-1234567}, implies
\begin{equation*}
X(t;s)-y_{s}=X(t;s)-X(s;s)+X(s;s)-y_{s}\in [C_{2},C_{2}+c_{\max}].
\end{equation*}
Now, if $x\in[X(t;s)-M,X(t;s)+M]$, then
\begin{equation*}
x-y_{s}=x-X(t;s)+X(t;s)-y_{s}\in[-M,M]+[C_{2},C_{2}+c_{\max}]\subset [C_{2}-M,C_{2}+c_{\max}+M].
\end{equation*}
In particular, there exists $c_{M}>0$ such that
\begin{equation*}
\sup_{x\in[X(t;s)-M,X(t;s)+M]}\phi'_{\min}(x-y_{s})\leq\sup_{x\in[C_{2}-M,C_{2}+c_{\max}+M]}\phi_{\min}'(x)\leq-c_{M},
\end{equation*}
since $\phi_{\min}'$ is continuous and negative everywhere. It then follows from \eqref{formula-for-derivative-2} that
\begin{equation*}
\sup_{s<0,0\leq t-s\leq 1}\sup_{x\in[X(t;s)-M,X(t;s)+M]}u_{x}(t,x;s)\leq-c_{M}e^{-(1-a)}.
\end{equation*}
In particular, \eqref{finite-time-steepness} follows. This completes the proof.
\end{proof}

For $\la\in(0,1)$, let $X_{\la}(t)$ be such that $u(t,X_{\la}(t))=\la$ for all $t\in\R$. It is well-defined by the monotonicity of $u(t,x)$ in $x$. As a simple consequence of Theorem \ref{thm-steepness}, we have
\begin{cor}\label{cor-interface-at-const-value}
For any $\la\in(0,1)$, $X_{\la}:\R\to\R$ is continuously differentiable and satisfies $\sup_{t\in\R}|\dot{X}_{\la}(t)|<\infty$.
\end{cor}
\begin{proof}
Let $\la\in(0,1)$. By Theorem \ref{thm-steepness} and the fact that $\sup_{t\in\R}|X_{\la}(t)-X(t)|<\infty$ due to Proposition\ref{prop-property-approximating-sol}$\rm(2)(ii)$, there exists some $\al_{\la}>0$ such that
\begin{equation}\label{steep-at-some-constant}
\sup_{t\in\R}u_{x}(t,X_{\la}(t))\leq-\al_{\la}.
\end{equation}
Then, since $u(t,X_{\la}(t))=\la$, implicit function theorem says that $X_{\la}(t)$ is continuously differentiable. Differentiating the equation $u(t,X_{\la}(t))=\la$ with respect to $t$, we find
\begin{equation*}
\dot{X}_{\la}(t)=-\frac{u_{t}(t,X_{\la}(t))}{u_{x}(t,X_{\la}(t))}.
\end{equation*}
The result then follows from \eqref{steep-at-some-constant} and the fact $\sup_{(t,x)\in\R\times\R}|u_{t}(t,x)|<\infty$.
\end{proof}


\section{Stability of transition fronts}\label{sec-stability}

In this section, we study the stability of transition fronts and prove Theorem \ref{thm-stability}. Throughout this section, we assume $\rm(H1)$-$\rm(H4)$. To this end, we first prove a Lemma.   Let $c_{\min}$, $c_{\max}$ be as in \eqref{tw-homo},
$M_1$ be as in \eqref{condition-1}, and $\Gamma=\Gamma_\alpha$ be as in \eqref{parameter-1}.

\begin{lem}\label{lem-technical}
\begin{itemize}
\item[\rm(i)] Let $I(r)=\int_{\R}J(x)e^{-rx}dx$ for $r\in\R$. Then,
\begin{equation*}
I(r)=1+\frac{I''(\tilde{r})}{2}r^{2}
\end{equation*}
for some $\tilde{r}=\tilde{r}(r)$ satisfying $|\tilde{r}|\leq|r|$.

\item[\rm(ii)]There exists $\alpha_0>0$ satisfying that for any
$0<\alpha\le \alpha_0$, there exists  $M_{2}=M_{2}(\al)>M_{1}+1$ such that
\begin{equation*}
|e^{\al(x-M_{1})}[J\ast\Ga](x)-1|\leq\frac{\al c_{\min}}{4},\quad\forall x\geq M_{2}.
\end{equation*}
In particular,
\begin{equation}\label{condition-2}
\big|[J\ast\Ga](x)-e^{-\al(x-M_{1})}\big|\leq\frac{\al c_{\min}}{4}e^{-\al(x-M_{1})},\quad\forall x\geq M_{2}.
\end{equation}
\end{itemize}
\end{lem}
\begin{proof}
$\rm(i)$ By $\rm(H2)$, $I(r)$ is well-defined for any $r\in\R$ and it is smooth in $r$. We see $I(0)=1$. Since
\begin{equation*}
I'(r)=-\int_{R}J(y)ye^{-ry}dy\to0\,\,\text{as}\,\,r\to0
\end{equation*}
due to the symmetry of $J$, we have $I'(0)=0$. The result then follows from second-order Taylor expansion at $r=0$.

$\rm(ii)$ Since $\Ga(x)=e^{-\al(x-M_{1})}$ for $x\geq M_{1}+1$, we deduce
\begin{equation}\label{an-equality-12345}
\begin{split}
&e^{\al(x-M_{1})}[J\ast\Ga](x)-1\\
&\quad\quad=e^{\al(x-M_{1})}\int_{-\infty}^{M_{1}+1}J(x-y)\Ga(y)dy+\int_{M_{1}+1}^{\infty}J(x-y)e^{\al(x-y)}dy-1
\end{split}
\end{equation}
For the first term on the right hand side of \eqref{an-equality-12345}, we see that since $J(x)$ decays faster than exponential functions by $\rm(H2)$ at $-\infty$, it is not hard to check that $e^{\al(x-M_{1})}\int_{-\infty}^{M_{1}+1}J(x-y)\Ga(y)dy\to0$ as $x\to\infty$. Notice this limit is locally  uniform in $\al\in[0,\infty)$. Thus, there exists $\tilde{M}_{2}=\tilde{M}_{2}(\al)>0$ such that
\begin{equation*}
\bigg|e^{\al(x-M_{1})}\int_{-\infty}^{M_{1}+1}J(x-y)\Ga(y)dy\bigg|\leq\frac{\al c_{\min}}{8},\quad\forall x\geq\tilde{M}_{2}.
\end{equation*}

For the last two terms on the right hand side of \eqref{an-equality-12345}, we have
\begin{equation*}
\begin{split}
\bigg|\int_{M_{1}+1}^{\infty}J(x-y)e^{\al(x-y)}dy-1\bigg|&=\bigg|\int_{M_{1}+1-x}^{\infty}J(y)e^{-\al y}dy-1\bigg|\\
&\leq\bigg|\int_{\R}J(y)e^{-\al y}dy-1\bigg|+\bigg|\int_{-\infty}^{M_{1}+1-x}J(y)e^{-\al y}dy\bigg|
\end{split}
\end{equation*}
By $\rm(i)$,  we conclude that there is $\alpha_0>0$ such that  for any $0<\alpha\le\alpha_0$,
$\big|\int_{\R}J(y)e^{-\al y}dy-1\big|\leq\frac{\al c_{\min}}{16}$. Since $\big|\int_{-\infty}^{M_{1}+1-x}J(y)e^{-\al y}dy\big|\to0$  locally uniformly in $\al\in[0,\infty)$ as $x\to\infty$, we can find some $\bar{M}_{2}=\bar{M}_{2}(\al)>0$ such that
\begin{equation*}
\bigg|\int_{-\infty}^{M_{1}+1-x}J(y)e^{-\al y}dy\bigg|\leq\frac{\al c_{\min}}{16},\quad\forall x\geq\bar{M}_{2}.
\end{equation*}
Hence,
\begin{equation*}
\bigg|\int_{M_{1}+1}^{\infty}J(x-y)e^{\al(x-y)}dy-1\bigg|\leq\frac{\al c_{\min}}{8},\quad\forall x\geq\bar{M}_{2}.
\end{equation*}
The result follows with $M_{2}=\max\{\tilde{M}_{2},\bar{M}_{2}\}$.
\end{proof}

Now, we prove Theorem \ref{thm-stability}.

\begin{proof}[Proof of Theorem \ref{thm-stability}]
$\rm(1)$ Let $\alpha_0$ be as in Lemma \ref{lem-technical}. For given $0<\alpha\le\alpha_0$, set
\begin{equation*}
C_{\rm steep}=C_{\rm steep}(\al):=-\sup_{t\in\R}\sup_{x\in[X(t)-M_{2},X(t)+M_{2}]}u_{x}(t,x)>0
\end{equation*}
by Theorem \ref{thm-steepness}. Set
\begin{equation}\label{parameter-2}
C_{f_{u}}:=\sup_{(t,u)\in\R\times[0,\infty)}|f_{u}(t,u)|\quad\text{and}\quad A=A(\al):=\frac{2C_{f_{u}}+1}{C_{\rm steep}}.
\end{equation}
Finally, set
\begin{equation}\label{parameter-3}
\ep_{0}=\ep_{0}(\al):=\min\bigg\{\frac{1-\tilde{\theta}}{2},\frac{\theta}{2},\frac{1}{4A},\frac{c_{\min}}{4A}\bigg\}\quad\text{and}\quad \om=\om(\al):=\min\bigg\{\tilde{\beta},\frac{\al c_{\min}}{4}\bigg\},
\end{equation}
where $\tilde{\beta}>0$ is as in $\rm(H4)$. Clearly, $\tilde{\beta}\leq C_{f_{u}}$.

We are going to prove $\rm(1)$ by constructing appropriate  sub-solution and super-solution. We first construct a sub-solution. Let
\begin{equation*}
u^{-}(t,x;t_{0})=u(t,x-\zeta^{-}(t))-q(t)\Ga(x-\zeta^{-}(t)-X(t)), \quad t\geq t_{0},\,\,x\in\R,
\end{equation*}
where
\begin{equation*}
\zeta^{-}(t)=\zeta_{0}^{-}-\frac{A\ep}{\om}(1-e^{-\om(t-t_{0})})\quad\text{and}\quad q(t)=\ep e^{-\om(t-t_{0})}.
\end{equation*}
Clearly, $\dot{\zeta}^{-}(t)=-Aq(t)$ and $\dot{q}(t)=-\om q(t)$. We claim that $u^{-}=u^{-}(t,x;t_{0})$ is a sub-solution, that is, $u^{-}_{t}\leq J\ast u^{-}-u^{-}+f(t,u^{-})$. To show this, we consider three cases.

\paragraph{\textbf{Case 1}} $x-\zeta^{-}(t)-X(t)\leq-M_{2}$. For such $x$, $u^{-}=u(t,x-\zeta^{-}(t))-q(t)$. We see
\begin{equation*}
\begin{split}
&u^{-}_{t}-[J\ast u^{-}-u^{-}]-f(t,u^{-})\\
&\quad\quad=u_{t}(t,x-\zeta^{-}(t))-\dot{\zeta}^{-}(t)u_{x}(t,x-\zeta^{-}(t))-\dot{q}(t)-[[J\ast u(t,\cdot-\zeta^{-}(t))](x)-u(t,x-\zeta^{-}(t))]\\
&\quad\quad\quad+[[J\ast\Ga(\cdot-\zeta^{-}(t)-X(t))](x)-1]q(t)-f(t,u^{-})\\
&\quad\quad=f(t,u(t,x-\zeta^{-}(t)))-f(t,u^{-})-\dot{\zeta}^{-}(t)u_{x}(t,x-\zeta^{-}(t))+\om q(t)\\
&\quad\quad\quad+[[J\ast\Ga(\cdot-\zeta^{-}(t)-X(t))](x)-1]q(t).
\end{split}
\end{equation*}

Notice $\dot{\zeta}^{-}(t)u_{x}(t,x-\zeta^{-}(t))\geq0$. We see that $u(t,x-\zeta^{-}(t))\geq\frac{1+\tilde{\theta}}{2}$ by the choice of $M_{1}$ in \eqref{condition-1} and $M_{2}$. Since $\ep\leq\ep_{0}\leq\frac{1-\tilde{\theta}}{2}$, there holds $u^{-}\geq\tilde{\theta}$. Thus, by $\rm(H4)$, we find
\begin{equation*}
f(t,u(t,x-\zeta^{-}(t)))-f(t,u^{-})\leq-\tilde{\beta}q(t).
\end{equation*}
Moreover, trivially $[J\ast\Ga(\cdot-\zeta^{-}(t)-X(t))](x)\leq1$, since $\Ga\in[0,1]$. Hence,
\begin{equation*}
u^{-}_{t}-[J\ast u^{-}-u^{-}]-f(t,u^{-})\leq-\tilde{\beta}q(t)+\om q(t)\leq0.
\end{equation*}

\paragraph{\textbf{Case 2}} $x-\zeta^{-}(t)-X(t)\geq M_{2}$. For such $x$, $u^{-}=u(t,x-\zeta^{-}(t))-q(t)e^{-\al(x-\zeta^{-}(t)-X(t)-M_{1})}$. We see
\begin{equation*}
\begin{split}
&u^{-}_{t}-[J\ast u^{-}-u^{-}]-f(t,u^{-})\\
&\quad\quad=u_{t}(t,x-\zeta^{-}(t))-\dot{\zeta}^{-}(t)u_{x}(t,x-\zeta^{-}(t))-[\dot{q}(t)+\al q(t)(\dot{\zeta}^{-}(t)+\dot{X}(t))]e^{-\al(x-\zeta^{-}(t)-X(t)-M_{1})}\\
&\quad\quad\quad-[[J\ast u(t,\cdot-\zeta^{-}(t))](x)-u(t,x-\zeta^{-}(t))]\\
&\quad\quad\quad+[[J\ast\Ga(\cdot-\zeta^{-}(t)-X(t))](x)-e^{-\al(x-\zeta^{-}(t)-X(t)-M_{1})}]q(t)-f(t,u^{-})\\
&\quad\quad=f(t,u(t,x-\zeta^{-}(t)))-f(t,u^{-}))-\dot{\zeta}^{-}(t)u_{x}(t,x-\zeta^{-}(t))\\
&\quad\quad\quad-[\dot{q}(t)+\al q(t)(\dot{\zeta}^{-}(t)+\dot{X}(t))]e^{-\al(x-\zeta^{-}(t)-X(t)-M_{1})}\\
&\quad\quad\quad+[[J\ast\Ga(\cdot-\zeta^{-}(t)-X(t))](x)-e^{-\al(x-\zeta^{-}(t)-X(t)-M_{1})}]q(t)
\end{split}
\end{equation*}

Again, $\dot{\zeta}^{-}(t)u_{x}(t,x-\zeta^{-}(t))\geq0$. We see that $u(t,x-\zeta^{-}(t))\leq\frac{\theta}{2}$ by the choice of $M_{1}$ and $M_{2}$, and therefore, $u^{-}\leq u(t,x-\zeta^{-}(t))\leq\frac{\theta}{2}$. Thus, $f(t,u(t,x-\zeta^{-}(t)))=0=f(t,u^{-})$. Moreover,
\begin{equation*}
\dot{q}(t)+\al q(t)(\dot{\zeta}^{-}(t)+\dot{X}(t))=(-\om-A\al q(t)+\al\dot{X}(t))q(t)\geq(-\om-A\al\ep_{0}+\al c_{\min})q(t).
\end{equation*}
Also, by \eqref{condition-2}, we have
\begin{equation*}
\begin{split}
&\bigg|[J\ast\Ga(\cdot-\zeta^{-}(t)-X(t))](x)-e^{-\al(x-\zeta^{-}(t)-X(t)-M_{1})}\bigg|\\
&\quad\quad=\bigg|\int_{\R}J(x-\zeta^{-}(t)-X(t)-y)\Ga(y)dy-e^{-\al(x-\zeta^{-}(t)-X(t)-M_{1})}\bigg|\\
&\quad\quad\leq\frac{\al c_{\min}}{4}e^{-\al(x-\zeta^{-}(t)-X(t)-M_{1})}.
\end{split}
\end{equation*}
It then follows that
\begin{equation*}
u^{-}_{t}-[J\ast u^{-}-u^{-}]-f(t,u^{-})\leq\bigg(\om+A\al\ep_{0}-\al c_{\min}+\frac{\al c_{\min}}{4}\bigg)q(t)e^{-\al(x-\zeta^{-}(t)-X(t)-M_{1})}\leq0.
\end{equation*}

\paragraph{\textbf{Case 3}} $|x-\zeta^{-}(t)-X(t)|\leq[-M_{2},M_{2}]$. We compute
\begin{equation*}
\begin{split}
&u^{-}_{t}-[J\ast u^{-}-u^{-}]-f(t,u^{-})\\
&\quad\quad=u_{t}(t,x-\zeta^{-}(t))-\dot{\zeta}^{-}(t)u_{x}(t,x-\zeta^{-}(t))\\
&\quad\quad\quad-\dot{q}(t)\Ga(x-\zeta^{-}(t)-X(t))+q(t)\Ga'(x-\zeta^{-}(t)-X(t))[\dot{\zeta}^{-}(t)+\dot{X}(t)]\\
&\quad\quad\quad-[[J\ast u(t,\cdot-\zeta^{-}(t))](x)-u(t,x-\zeta^{-}(t))]\\
&\quad\quad\quad+[[J\ast\Ga(\cdot-\zeta^{-}(t)-X(t))](x)-\Ga(x-\zeta^{-}(t)-X(t))]q(t)-f(t,u^{-})\\
&\quad\quad=f(t,u(t,x-\zeta^{-}(t)))-f(t,u^{-})+Aq(t)u_{x}(t,x-\zeta^{-}(t))\\
&\quad\quad\quad+\om q(t)\Ga(x-\zeta^{-}(t)-X(t))+q(t)\Ga'(x-\zeta^{-}(t)-X(t))[\dot{X}(t)-Aq(t)]\\
&\quad\quad\quad+[[J\ast\Ga(\cdot-\zeta^{-}(t)-X(t))](x)-\Ga(x-\zeta^{-}(t)-X(t))]q(t).
\end{split}
\end{equation*}
We see that
\begin{equation*}
\begin{split}
|f(t,u(t,x-\zeta^{-}(t)))-f(t,u^{-})|&\leq C_{f_{u}}q(t),\\
Aq(t)u_{x}(t,x-\zeta^{-}(t))&\leq-AC_{\rm steep}q(t),\\
\om q(t)\Ga(x-\zeta^{-}(t)-X(t))&\leq\om q(t),\\
q(t)\Ga'(x-\zeta^{-}(t)-X(t))[\dot{X}(t)-Aq(t)]&\leq q(t)\Ga'(x-\zeta^{-}(t)-X(t))[c_{\min}-A\ep_{0}]\leq0,\\
[[J\ast\Ga(\cdot-\zeta^{-}(t)-X(t))](x)-\Ga(x-\zeta^{-}(t)-X(t))]q(t)&\leq q(t).
\end{split}
\end{equation*}
It then follows that
\begin{equation*}
u^{-}_{t}-[J\ast u^{-}-u^{-}]-f(t,u^{-})\leq(C_{f_{u}}-AC_{\rm steep}+\om+1)q(t)\leq0.
\end{equation*}

Hence, we have shown $u^{-}_{t}-[J\ast u^{-}-u^{-}]-f(t,u^{-})\leq0$, that is, $u^{-}$ is a sub-solution. By the first inequality in \eqref{initial-condition} and comparison principle, we conclude that
\begin{equation}\label{lower-bound}
u(t,x-\zeta^{-}(t))-q(t)\Ga(x-\zeta^{-}(t)-X(t))=u^{-}(t,x;t_{0})\leq u(t,x;t_{0},u_{0}).
\end{equation}

For the super-solution, we set
\begin{equation*}
u^{+}(t,x;t_{0})=u(t,x-\zeta^{+}(t))+q(t)\Ga(x-\zeta^{+}(t)-X(t)),\quad t\geq t_{0},\,\,x\in\R,
\end{equation*}
where
\begin{equation*}
\zeta^{+}(t)=\zeta_{0}^{+}+\frac{A\ep}{\om}(1-e^{-\om(t-t_{0})})\quad\text{and}\quad q(t)=\ep e^{-\om(t-t_{0})}.
\end{equation*}
The proof of $u^{+}=u^{+}(t,x;t_{0})$ being a super-solution, that is, $u^{+}_{t}\geq J\ast u^{+}-u^{+}+f(t,u^{+})$, follows from arguments for the sub-solution. We outline the proof for completeness.

\paragraph{\textbf{Case 1}} $x-\zeta^{+}(t)-X(t)\leq-M_{2}$. We compute
\begin{equation*}
\begin{split}
&u_{t}^{+}-[J\ast u^{+}-u^{+}]-f(t,u^{+})\\
&\quad\quad=f(t,u(t,x-\zeta^{+}(t)))-f(t,u^{+})-\dot{\zeta}^{+}(t)u_{x}(t,x-\zeta^{+}(t))-\om q(t)\\
&\quad\quad\quad-[[J\ast\Ga(\cdot-\zeta^{+}(t)-X(t))](x)-1]q(t)\\
&\quad\quad\geq\tilde{\beta}q(t)-\om q(t)\geq0.
\end{split}
\end{equation*}

\paragraph{\textbf{Case 2}} $x-\zeta^{+}(t)-X(t)\geq M_{2}$. We compute
\begin{equation*}
\begin{split}
&u_{t}^{+}-[J\ast u^{+}-u^{+}]-f(t,u^{+})\\
&\quad\quad=f(t,u(t,x-\zeta^{+}(t)))-f(t,u^{+})-\dot{\zeta}^{+}(t)u_{x}(t,x-\zeta^{+}(t))\\
&\quad\quad\quad+[\al(\dot{\zeta}^{+}(t)+\dot{X}(t))-\om]q(t)e^{-\al(x-\zeta^{+}(t)-X(t)-M_{1})}\\
&\quad\quad\quad-[[J\ast\Ga(\cdot-\zeta^{+}(t)-X(t))](x)-e^{-\al(x-\zeta^{+}(t)-X(t)-M_{1})}]q(t).
\end{split}
\end{equation*}
We see $u(t,x-\zeta^{+}(t))\leq\frac{\theta}{2}$, and therefore, $u^{+}\leq\theta$ since $\ep_{0}\leq\frac{\theta}{2}$. In particular, $f(t,u(t,x-\zeta^{+}(t)))-f(t,u^{+})=0$. Since $-\dot{\zeta}^{+}(t)u_{x}(t,x-\zeta^{+}(t))\geq0$,
\begin{equation*}
\al(\dot{\zeta}^{+}(t)+\dot{X}(t))-\om\geq\al Aq(t)+\al c_{\min}-\om\geq0
\end{equation*}
and
\begin{equation*}
[J\ast\Ga(\cdot-\zeta^{+}(t)-X(t))](x)-e^{-\al(x-\zeta^{+}(t)-X(t)-M_{1})}\leq\frac{\al c_{\min}}{4}e^{-\al(x-\zeta^{-}(t)-X(t)-M_{1})},
\end{equation*}
we have
\begin{equation*}
u_{t}^{+}-[J\ast u^{+}-u^{+}]-f(t,u^{+})\geq\bigg(\al Aq(t)+\al c_{\min}-\om-\frac{\al c_{\min}}{4}\bigg)q(t)e^{-\al(x-\zeta^{+}(t)-X(t)-M_{1})}\geq0.
\end{equation*}

\paragraph{\textbf{Case 3}} $|x-\zeta^{+}(t)-X(t)|\leq M_{2}$. We compute
\begin{equation*}
\begin{split}
&u_{t}^{+}-[J\ast u^{+}-u^{+}]-f(t,u^{+})\\
&\quad\quad=f(t,u(t,x-\zeta^{+}(t)))-f(t,u^{+})-\dot{\zeta}^{+}(t)u_{x}(t,x-\zeta^{+}(t))\\
&\quad\quad\quad-\om q(t)\Ga(x-\zeta^{+}(t)-X(t))-q(t)\Ga'(x-\zeta^{+}(t)-X(t))(Aq(t)+\dot{X}(t))\\
&\quad\quad\quad-[[J\ast\Ga(\cdot-\zeta^{+}(t)-X(t))](x)-\Ga(x-\zeta^{+}(t)-X(t))]q(t).
\end{split}
\end{equation*}
We see that
\begin{equation*}
\begin{split}
f(t,u(t,x-\zeta^{+}(t)))-f(t,u^{+})&\geq-C_{f_{u}}q(t),\\
-\dot{\zeta}^{+}(t)u_{x}(t,x-\zeta^{+}(t))&\geq0,\\
-\om q(t)\Ga(x-\zeta^{+}(t)-X(t))&\geq-\om q(t),\\
-q(t)\Ga'(x-\zeta^{+}(t)-X(t))(Aq(t)+\dot{X}(t))&\geq0,\\
-[[J\ast\Ga(\cdot-\zeta^{+}(t)-X(t))](x)-\Ga(x-\zeta^{+}(t)-X(t))]q(t)&\geq-q(t).
\end{split}
\end{equation*}
It then follows that
\begin{equation*}
u_{t}^{+}-[J\ast u^{+}-u^{+}]-f(t,u^{+})\geq(AC_{\rm steep}-C_{f_{u}}-\om-1)q(t)\geq0.
\end{equation*}

Hence, $u^{+}$ is a super-solution. By the second inequality in \eqref{initial-condition} and comparison principle, we conclude that
\begin{equation}\label{upper-bound}
u(t,x;t_{0},u_{0})\leq u^{+}(t,x;t_{0})=u(t,x-\zeta^{+}(t))+q(t)\Ga(x-\zeta^{+}(t)-X(t)).
\end{equation}
The result then follows from \eqref{lower-bound} and \eqref{upper-bound}.

$\rm(2)$ Note first that  there is $0<\alpha=\alpha(\beta_0)\le\alpha_0$ satisfying that  for any $\ep\in(0,\ep_{0}(\alpha)]$, there exists $\zeta_{0}^{\pm}=\zeta_{0}^{\pm}(\ep,u_{0})\in\R$ with $\zeta_{0}^{-}<\zeta_{0}^{+}$ such that
\begin{equation}\label{initial-condition-1234567890}
u(t_{0},x-\zeta_{0}^{-})-\ep\Ga_\alpha(x-\zeta_{0}^{-}-X(t_{0}))\leq u_{0}(x)\leq u(t_{0},x-\zeta_{0}^{+})+\ep\Ga_\alpha(x-\zeta_{0}^{+}-X(t_{0})).
\end{equation}
We then conclude $\rm(2)$ by applying $\rm(1)$ and noticing that $\lim_{t\to\infty}\zeta^{\pm}(t)$ exist and $\Ga\in[0,1]$.
\end{proof}

Note that the proof of Theorem \ref{thm-stability}$\rm(2)$ does not depend explicitly on the condition on $u_{0}$ as in the statement of Theorem \ref{thm-stability}$\rm(2)$; instead, it only needs \eqref{initial-condition-1234567890}. This observation allows us to prove the following corollary, which generalizes Theorem \ref{thm-stability}$\rm(2)$ to more general initial data.

\begin{cor}\label{cor-stability-12345}
Let $u(t,x)$ and $X(t)$ be as in Proposition\ref{prop-property-approximating-sol}(2). Let $\beta_{0}>0$. Suppose $t_{0}\in\R$ and $\tilde{u}_{0}\in C_{\rm unif}^{b}(\R,\R)$ satisfy
\begin{equation*}
\begin{cases}
\tilde{u}_{0}:\R\to[0,1],\quad \liminf_{x\to-\infty}\tilde{u}_{0}(x)>\theta;\\
\exists C>0\,\,\text{s.t.}\,\,|\tilde{u}_{0}-u(t_{0},x)|\leq Ce^{-\beta_{0}(x-X(t_{0}))}\,\,\text{for}\,\,x\in\R.
\end{cases}
\end{equation*}
Then, there exist $\om>0$  and $\tilde\epsilon_0>0$ such that for any $\ep\in(0,\tilde\ep_{0}]$, there are $\zeta_{\pm}=\zeta_{\pm}(\ep,u_{0})\in\R$ and $t_{1}=t_{1}(\ep,u_{0})$ such that
\begin{equation*}
u(t,x-\zeta_{-})-\ep e^{-\om(t-t_{1})}\leq u(t,x;t_{0},\tilde{u}_{0})\leq u(t,x-\zeta_{+})+\ep e^{-\om(t-t_{1})}
\end{equation*}
for all $x\in\R$ and $t\geq t_{1}$.
\end{cor}
\begin{proof}
The idea is that we allow the solution $u(t,x;t_{0},\tilde{u}_{0})$ to evolve for some time. Due to the asymptotical stability of $1$, it will develop into some shape satisfying \eqref{initial-condition-1234567890}. Then, we apply Theorem \ref{thm-stability}$\rm(2)$ at that time to conclude the result.

Modifying $\tilde{u}_{0}$ near $-\infty$, we can find $u_{0}\in C_{\rm unif}^{b}(\R,\R)$ satisfying $u_{0}\geq\tilde{u}_{0}$ and
\begin{equation*}
\begin{cases}
u_{0}:\R\to[0,1],\quad u_{0}(-\infty)=1;\\
\exists C>0\,\,\text{s.t.}\,\,|u_{0}-u(t_{0},x)|\leq Ce^{-\beta_{0}(x-X(t_{0}))}\,\,\text{for}\,\,x\in\R.
\end{cases}
\end{equation*}
In particular, we can apply Theorem \ref{thm-stability}$\rm(2)$ to $u_{0}$ to conclude that
\begin{equation}\label{one-sided-estimate}
u(t,x;t_{0},u_{0})\leq u(t,x-\zeta^{+})+q(t)\Ga_{\al}(x-\zeta^{+}-X(t)),
\end{equation}
where $q(t)=e^{-\om(t-t_{0})}$ and $\Ga_{\al}$ is the same as in the proof of Theorem \ref{thm-stability}$\rm(2)$. Notice \eqref{one-sided-estimate} holds for some $0<\al\le\alpha_0$. Since $u_{0}\geq\tilde{u}_{0}$, we have from comparison and \eqref{one-sided-estimate} that
\begin{equation}\label{one-sided-estimate-1}
u(t,x;t_{0},\tilde{u}_{0})\leq u(t,x-\zeta^{+})+q(t)\Ga_{\al}(x-\zeta^{+}-X(t)).
\end{equation}
Thus, for a fixed small $\ep>0$, we can find some $t_{1}=t_{1}(\ep)\gg t_{0}$ such that $q(t)\leq\ep$, and then,
\begin{equation}\label{one-sided-estimate-2}
u(t_{1},x;t_{0},\tilde{u}_{0})\leq u(t_{1},x-\zeta^{+})+\ep\Ga_{\al}(x-\zeta^{+}-X(t_{1})).
\end{equation}

Next, we construct an appropriate lower bound for $u(t_{1},x;t_{0},\tilde{u}_{0})$. This actually follows from the asymptotic stability of the equilibrium $1$. More precisely, since $\liminf_{x\to-\infty}\tilde{u}_{0}(x)>\theta$, there exist $\la_{0}\in(\theta,\liminf_{x\to-\infty}\tilde{u}_{0}(x))$ and a function $\bar{u}_{0}\in C_{\rm unif}^{b}(\R,\R)$ satisfying
\begin{equation*}
\exists\,x_{1}<x_{2}\,\,\text{s.t.}\,\, \bar{u}_{0}(x)=
\begin{cases}
\la_{0}\quad&\text{if}\quad x\leq x_{1},\\
0\quad&\text{if}\quad x\geq x_{2}
\end{cases}
\end{equation*}
such that $\bar{u}_{0}\leq\tilde{u}_{0}$. Now, we consider the solution $u_{B}(t,x;\bar{u}_{0})$ with initial data $u_{B}(0,\cdot;\bar{u}_{0})=\bar{u}_{0}$ of the following homogeneous equation
\begin{equation}\label{eqn-homo-bistable}
u_{t}=J\ast u-u+f_{B}(u)
\end{equation}
where $f_{B}:[0,1]\to\R$ is a bistable nonlinearity satisfying the following conditions
\begin{equation*}
\begin{cases}
f_{B}\in C^{2}([0,1]),\,\,f_{B}(0)=0,\,\,f_{B}(\theta)=0,\,\,f_{B}(1)=0,\\
f_{B}'(0)<0,\,\,f_{B}'(1)<0,\\
f_{B}(u)<0\,\,\text{for}\,\,u\in(0,\theta),\,\,0<f_{B}(u)\leq f_{\min}(u)\,\,\text{for}\,\,u\in(\theta,1),\\
\int_{0}^{1}f_{B}(u)du>0\,\,\text{and}\,\,1+f_{B}'(u)>0\,\,\text{for}\,\,u\in[0,1].
\end{cases}
\end{equation*}
Let $c_{B}>0$ be the unique speed of the traveling waves of \eqref{eqn-homo-bistable}, and we fix some profile $\phi_{B}$. Since $f_{B}\leq f_{\min}\leq f(t,u)$ on $[0,1]$, we conclude from the comparison principle that
\begin{equation*}
u_{B}(t-t_{0},x;\bar{u}_{0})\leq u(t,x;t_{0},u_{0}).
\end{equation*}
It is known (see \cite[Theorem 4.2]{BaFiReWa97}) that there exists $\zeta_{B}^{\pm}\in\R$, $\ep_{B}>0$ and $\om_{B}>0$ such that
\begin{equation*}
\phi_{B}(x-c_{B}t-\zeta_{B}^{-})-\ep_{B}e^{-\om_{B}(t-t_{0})}\leq u_{B}(t-t_{0},x;\bar{u}_{0})\leq\phi_{B}(x-c_{B}t-\zeta_{B}^{+})+\ep_{B}e^{-\om_{B}(t-t_{0})}.
\end{equation*}
In particular, $u_{B}(t-t_{0},-\infty;\bar{u}_{0})$, and hence, $u(t,x;t_{0},u_{0})$, approaches to $1$ exponentially fast. Thus, making $\al>0$ so small that $\om$ is small and choosing $t_{1}$ larger if necessary, we can guarantee that
\begin{equation*}
u(t_{1},-\infty;t_{0},u_{0})\geq 1-\frac{\ep}{2}.
\end{equation*}
Thus, choosing $\al>0$ further small if necessary, we can find $\zeta^{-}\in\R$ such that
\begin{equation}\label{one-sided-estimate-3}
u(t_{1},x-\zeta^{-})-\ep\Ga_{\al}(x-\zeta^{-}-X(t_{1}))\leq u(t_{1},x;t_{0},\tilde{u}_{0}).
\end{equation}

Finally, in the presence of \eqref{one-sided-estimate-1} and \eqref{one-sided-estimate-3}, we can apply Theorem \ref{thm-stability}$\rm(2)$ to $u(t_{1},x;t_{0},\tilde{u}_{0})$ to conclude the result.
\end{proof}


\section{Asymptotic stability of transition fronts}\label{sec-asymptotic-stability}

In this section, we study the asymptotic stability of $u(t,x)$ and prove Theorem \ref{thm-asymptotic-stability}. We assume $\rm(H1)$-$\rm(H4)$ throughout this section.

We first prove two lemmas. The first one concerns the exponential decay of $u_{x}(t,x+X(t))$ at $\pm\infty$.

\begin{lem}\label{lem-decaying-space-derivative}
There exist $\tilde{c}_{\pm}>0$, $\tilde{C}_{\pm}>0$ and $\tilde{h}_{\pm}>0$ such that
\begin{equation*}
\begin{split}
0> u_{x}(t,x)&\geq -\tilde{C}_{+}e^{-\tilde{c}_{+}(x-X(t)-\tilde{h}_{+})},\quad\forall x\geq X(t)+\tilde{h}_{+},\\
0> u_{x}(t,x)&\geq -\tilde{C}_{-}e^{\tilde{c}_{-}(x-X(t)+\tilde{h}_{-})},\quad\forall x\leq X(t)-\tilde{h}_{-}
\end{split}
\end{equation*}
for all $t\in\R$.
\end{lem}
\begin{proof}
We prove the first estimate for $u_{x}(t,x)$. By monotonicity, $u_{x}(t,x)<0$. Since $X(t;s)$ and $u_{x}(t,x;s)$ converge locally uniformly to $X(t)$ and $u_{x}(t,x)$, respectively, it suffices to show
\begin{equation}\label{decaying-space-derivative-approximate-sol}
u_{x}(t,x;s)\geq-\tilde{C}e^{-\tilde{c}(x-X(t;s)-\tilde{h})},\quad\forall x\geq X(t;s)+\tilde{h}
\end{equation}
for all $s<0$, $t\geq s$. To this end, we set
\begin{equation*}
\tilde{C}=\sup_{s<0,t\geq s\atop x\in\R}|u_{x}(t,x;s)|\quad\text{and}\quad\tilde{h}\geq\sup_{s<0,t\geq s}|X(t;s)-X_{\theta}(t;s)|.
\end{equation*}
By the choice of $\tilde{h}$, we have $f(t,u(t,x;s))=0$ for $x\geq X(t;s)+\tilde{h}$. Since $u_{x}(t,x;s)$ satisfies $(u_{x})_{t}=J\ast u_{x}-u_{x}+f_{u}(t,u(t,x;s))u_{x}$, we see that
$u_{x}(t,x;s)$ satisfies
\begin{equation}\label{comparison-1234567}
(u_{x})_{t}=J\ast u_{x}-u_{x},\quad x\geq X(t;s)+\tilde{h}.
\end{equation}

Define
\begin{equation*}
N[v]=v_{t}-[J\ast v-v].
\end{equation*}
We compute
\begin{equation*}
\begin{split}
N[-\tilde{C}e^{-c(x-X(t)-\tilde{h})}]&=-\tilde{C}e^{-c(x-X(t)-\tilde{h})}\bigg[c\dot{X}(t;s)-\int_{\R}J(y)e^{cy}dy+1\bigg]\\
&\leq-\tilde{C}e^{-c(x-X(t)-\tilde{h})}\bigg[cc_{\min}-\int_{\R}J(y)e^{cy}dy+1\bigg].
\end{split}
\end{equation*}
Setting $g(c)=cc_{\min}-\int_{\R}J(y)e^{cy}dy+1$, we see $g(0)=0$ and $g'(c)=c_{\min}-\int_{\R}yJ(y)e^{cy}dy$. Since $c_{\min}>0$ and $\int_{\R}yJ(y)e^{cy}dy\to0$ as $c\to0$ by the symmetry of $J$, we are able to find $\tilde{c}>0$ such that $g(\tilde{c})>0$. It then follows that
$N[-\tilde{C}e^{-\tilde{c}(x-X(t)-\tilde{h})}]\leq0$. In particular, by \eqref{comparison-1234567}, we have
\begin{equation}\label{comparsion-123-1}
N[u_{x}]=0\geq N[-\tilde{C}e^{-\tilde{c}(x-X(t)-\tilde{h})}],\quad x\geq X(t;s)+\tilde{h},\quad t\geq s.
\end{equation}
Moreover, we trivially have
\begin{equation}\label{comparsion-123-2}
u_{x}(t,x;s)\geq-\tilde{C}\geq-\tilde{C}e^{-\tilde{c}(x-X(t)-\tilde{h})},\quad x\leq X(t;s)+\tilde{h},\quad t\geq s.
\end{equation}
Also, at the initial moment $s$, choosing $\tilde{c}$ smaller and $\tilde{h}$ larger (if necessary), we have
\begin{equation}\label{comparsion-123-3}
u_{x}(s,x;s)=\phi_{\min}'(x-y_{s})\geq-\tilde{C}e^{-\tilde{c}(x-X(s)-\tilde{h})}.
\end{equation}
We then conclude from \eqref{comparsion-123-1}, \eqref{comparsion-123-2}, \eqref{comparsion-123-3} and the comparison principle (see Proposition \ref{prop-app-comparison}) that \eqref{decaying-space-derivative-approximate-sol} holds. We point out that the above arguments work due to the fact that $\dot{X}(t;s)\geq c_{\min}>0$.

For the second estimate for $u_{x}(t,x)$, we notice that if we choose $\hat{h}$ be such that
\begin{equation*}
\hat{h}\geq\sup_{s<0,t\geq s}|X(t;s)-X_{\tilde{\theta}}(t;s)|,
\end{equation*}
where $\tilde{\theta}$ is as in $\rm(H4)$. Then, $f_{u}(t,u(t,x;s))\leq-\tilde{\beta}$ for $x\leq X(t;s)-\hat{h}$. It then follows that $u_{x}(t,x;s)$ satisfies
\begin{equation*}
(u_{x})_{t}\geq J\ast u_{x}-u_{x}-\tilde{\beta}u_{x},\quad x\leq X(t;s)-\hat{h}.
\end{equation*}
The rest of the proof then follows from similar arguments as above if we consider
\begin{equation*}
N[v]=v_{t}-[J\ast v-v]+\tilde{\beta}v.
\end{equation*}
This completes the proof.
\end{proof}

The second lemma, improving Theorem \ref{thm-stability}$\rm(1)$,  is the key to Theorem \ref{thm-asymptotic-stability}. Shall not cause any confusion with $u(t,x;s)$, we will use $u(t,x;t_{0})$ to denote a solution of \eqref{main-eqn} with initial condition at time $t_{0}$. Recall $\al>0$ is small, and $\Ga=\Ga_{\al}$, $A=A(\al)$, $\ep_{0}=\ep_{0}(\al)$ and $\om=\om(\al)$ are as in \eqref{parameter-1}, \eqref{parameter-2} and \eqref{parameter-3}.

\begin{lem}\label{lem-key-asymptotic-stability}
Suppose there exist $\zeta\in\R$, $\de>0$ and $\ep\in(0,\ep_{0}]$ such that
\begin{equation}\label{initial-condition-key-lem}
u(\tau,x-\zeta)-\ep\Ga(x-\zeta-X(\tau))\leq u(\tau,x;t_{0})\leq u(\tau,x-\zeta-\de)+\ep\Ga(x-\zeta-\de-X(\tau))
\end{equation}
for some $\tau\geq t_{0}$. Then, there exist large $\si=\si(\al)>0$ and small $\tilde{\ep}=\tilde{\ep}(\al,\ep_{0})>0$ such that
\begin{equation*}
\begin{split}
&u(t,x-\zeta(t))-q(t)\Ga(x-\zeta(t)-X(t))\\
&\quad\quad\leq u(t,x;t_{0})\leq u(t,x-\zeta(t)-\de(t))-q(t)\Ga(x-\zeta(t)-\de(t)-X(t))
\end{split}
\end{equation*}
for all $t\geq\tau+\si$, where
\begin{equation*}
\begin{split}
\zeta(t)&\in[\zeta-\frac{2A\ep}{\om},\zeta+\tilde{\ep}\min\{1,\de\}],\\
0\leq\de(t)&\leq\de-\tilde{\ep}\min\{1,\de\}+\frac{4A\ep}{\om},\\
0\leq q(t)&\leq(\frac{\ep}{2}+\tilde{C}\tilde{\ep}\min\{1,\de\})e^{-\om(t-\tau-\si)},
\end{split}
\end{equation*}
where $\tilde{C}>0$ is some constant and $\tilde{C}\tilde{\ep}\leq\frac{\ep_{0}}{2}$.
\end{lem}
\begin{proof}
Applying Theorem \ref{thm-stability}$\rm(1)$ to \eqref{initial-condition-key-lem}, we find
\begin{equation}\label{known-estimate}
\begin{split}
&u(t,x-\zeta_{\tau}^{-}(t))-q_{\tau}(t)\Ga(x-\zeta_{\tau}^{-}(t)-X(t))\\
&\quad\quad\leq u(t,x;t_{0})\leq u(t,x-\zeta_{\tau}^{+}(t)-\de)+q_{\tau}(t)\Ga(x-\zeta_{\tau}^{+}(t)-\de-X(t))
\end{split}
\end{equation}
for all $t\geq \tau$, where $\zeta_{\tau}^{\pm}(t)=\zeta\pm\frac{A\ep}{\om}(1-e^{-\om(t-\tau)})$ and $q_{\tau}(t)=\ep e^{-\om(t-\tau)}$.

We modify \eqref{known-estimate} at the moment $t=\tau+\si$ for some $\si>0$ to be chosen to obtain a new estimate for $u(\tau+\si,x;t_{0})$, and then apply Theorem \ref{thm-stability}(1) to this new estimate to conclude the result of the lemma. To this end, we set
\begin{equation*}
\tilde{\de}=\min\{\de,1\}
\quad\text{and}\quad \tilde{C}_{\rm steep}=\frac{1}{2}\sup\big\{u_{x}(t,x)\big||x-X(t)|\leq 2,\,\, t\geq t_{0}\big\}<0.
\end{equation*}
Then, for $t\geq t_{0}$, we deduce from Taylor expansion that
\begin{equation*}
\int_{X(t)-\frac{1}{2}}^{X(t)+\frac{1}{2}}\big[u(t,y-\tilde{\de})-u(t,y)\big]dy\geq-2\tilde{C}_{\rm steep}\tilde{\de}.
\end{equation*}
In particular, at the moment $t=\tau$, either
\begin{equation}\label{either-or-1}
\int_{X(\tau)-\frac{1}{2}}^{X(\tau)+\frac{1}{2}}\big[u(\tau,y-\tilde{\de})-u(\tau,y+\zeta;t_{0})\big]dy\geq-\tilde{C}_{\rm steep}\tilde{\de}
\end{equation}
or
\begin{equation}\label{either-or-2}
\int_{X(\tau)-\frac{1}{2}}^{X(\tau)+\frac{1}{2}}\big[u(\tau,y+\zeta;t_{0})-u(\tau,y)\big]dy\geq-\tilde{C}_{\rm steep}\tilde{\de}
\end{equation}
must be the case.

We first consider the problem when \eqref{either-or-2} holds. We are about to establish an appropriate lower bound for
\begin{equation*}
u(\tau+\si,x;t_{0})-u(\tau+\si,x-\zeta_{\tau}^{-}(\tau+\si)-\tilde{\ep}\tilde{\de}),
\end{equation*}
where $\tilde{\ep}>0$ and $\si>0$ are to be chosen. To do so, let $M>0$ be a large number to be chosen, and consider three cases: $\rm(i)$ $x-\zeta-X(\tau)\in[-M,M]$; $\rm(ii)$ $x-\zeta-X(\tau)\leq-M$; $\rm(iii)$ $x-\zeta-X(\tau)\geq-M$.

\textbf{Case $\rm(i)$. $x-\zeta-X(\tau)\in[-M,M]$.} We write
\begin{equation}\label{left-region}
\begin{split}
&u(\tau+\si,x;t_{0})-u(\tau+\si,x-\zeta_{\tau}^{-}(\tau+\si)-\tilde{\ep}\tilde{\de})\\
&\quad\quad=\big[u(\tau+\si,x;t_{0})-u(\tau+\si,x-\zeta_{\tau}^{-}(\tau+\si))\big]\\
&\quad\quad\quad+\big[u(\tau+\si,x-\zeta_{\tau}^{-}(\tau+\si))-u(\tau+\si,x-\zeta_{\tau}^{-}(\tau+\si)-\tilde{\ep}\tilde{\de})\big].
\end{split}
\end{equation}
For the first difference on the right hand side of \eqref{left-region}, we argue
\begin{equation*}
\begin{split}
&u(\tau+\si,x;t_{0})-u(\tau+\si,x-\zeta_{\tau}^{-}(\tau+\si))+q_{\tau}(\tau+\si)\Ga(x-\zeta_{\tau}^{-}(\tau+\si)-X(\tau+\si))\\
&\quad\quad=u(\tau+\si,x;t_{0})\\
&\quad\quad\quad-\big[u(\tau+\si,x-\zeta+\frac{A\ep}{\om}(1-e^{-\om\si}))-q_{\tau}(\tau+\si)\Ga(x-\zeta+\frac{A\ep}{\om}(1-e^{-\om\si})-X(\tau+\si))\big]\\
&\quad\quad=u(\tau+\si,y+\zeta;t_{0})\\
&\quad\quad\quad-\big[u(\tau+\si,y+\frac{A\ep}{\om}(1-e^{-\om\si}))-q_{\tau}(\tau+\si)\Ga(y+\frac{A\ep}{\om}(1-e^{-\om\si})-X(\tau+\si))\big]\\
&\quad\quad\quad\quad\quad\quad\quad\quad\quad\quad\quad\quad\quad\quad\quad\quad\quad\quad\quad\quad\quad\quad\quad\quad(\text{by}\,\,y=x-\zeta\in X(\tau)+[-M,M])\\
&\quad\quad=u(\tau+\si,y+\zeta;t_{0})-\tilde{u}(\tau+\si,y)\\
&\quad\quad(\text{where}\,\,\tilde{u}(t,y)=u(t,y+\frac{A\ep}{\om}(1-e^{-\om(t-\tau)}))-q_{\tau}(t)\Ga(y+\frac{A\ep}{\om}(1-e^{-\om(t-\tau)})-X(t)))\\
&\quad\quad\geq C(\si,M)\int_{X(\tau)-\frac{1}{2}}^{X(\tau)+\frac{1}{2}}\big[u(\tau,y+\zeta;t_{0})-\tilde{u}(\tau,y)\big]dy\\
&\quad\quad\geq C(\si,M)\int_{X(\tau)-\frac{1}{2}}^{X(\tau)+\frac{1}{2}}\big[u(\tau,y+\zeta;t_{0})-u(\tau,y)\big]dy\quad(\text{by}\,\,\tilde{u}(\tau,y)\leq u(\tau,y))\\
&\quad\quad\geq-C(\si,M)\tilde{C}_{\rm steep}\tilde{\de}\quad (\text{by \eqref{either-or-2}}),
\end{split}
\end{equation*}
where the first inequality follows as in the proof of Lemma \ref{lem-tech-1234567}. In fact, we know $u(t,y+\zeta;t_{0})$ is a solution of $v_{t}=J\ast v-v+f(t,v)$, while $\tilde{u}(t,y)$ is a subsolution by the proof of Theorem \ref{thm-stability}. Moreover, $u(t,y+\zeta;t_{0})\geq\tilde{u}(t,y)$ by \eqref{known-estimate}. Based on these information, we can repeat the arguments in the proof of Lemma \ref{lem-tech-1234567} to conclude the inequality. Here, $C(t-\tau,M)>0$ satisfies $C(t-\tau,M)\to0$ polynomially as $t-\tau\to0$ and exponentially as $t-\tau\to\infty$. Thus, we have shown
\begin{equation}\label{left-region-1}
\begin{split}
&u(\tau+\si,x;t_{0})-u(\tau+\si,x-\zeta_{\tau}^{-}(\tau+\si))\\
&\quad\quad\geq-C(\si,M)\tilde{C}_{\rm steep}\tilde{\de}-q_{\tau}(\tau+\si)\Ga(x-\zeta_{\tau}^{-}(\tau+\si)-X(\tau+\si)).
\end{split}
\end{equation}

For the second difference on the right hand side of \eqref{left-region}, Taylor expansion gives
\begin{equation*}
u(\tau+\si,x-\zeta_{\tau}^{-}(\tau+\si))-u(\tau+\si,x-\zeta_{\tau}^{-}(\tau+\si)-\tilde{\ep}\tilde{\de})=u_{x}(\tau+\si,x-\zeta+\frac{A\ep}{\om}(1-e^{-\om\si})-x_{*})\tilde{\ep}\tilde{\de},
\end{equation*}
where $x_{*}\in[0,\tilde{\ep}\tilde{\de}]\subset[0,1]$. Setting
\begin{equation}\label{aux-ep}
\tilde{\ep}=\tilde{\ep}(\si,M):=\min\bigg\{1,\frac{-\tilde{C}_{\rm steep}C(\si,M)}{\sup_{(t,x)\in\R\times\R}|u_{x}(t,x)|}\bigg\}>0,
\end{equation}
we deduce
\begin{equation}\label{left-region-2}
u(\tau+\si,x-\zeta_{\tau}^{-}(\tau+\si))-u(\tau+\si,x-\zeta_{\tau}^{-}(\tau+\si)-\tilde{\ep}\tilde{\de})\geq C(\si,M)\tilde{C}_{\rm steep}\tilde{\de}.
\end{equation}
It then follows from \eqref{left-region}, \eqref{left-region-1} and \eqref{left-region-2} that
\begin{equation}\label{left-region-12345}
u(\tau+\si,x;t_{0})-u(\tau+\si,x-\zeta_{\tau}^{-}(\tau+\si)-\tilde{\ep}\tilde{\de})\geq-q_{\tau}(\tau+\si)\Ga(x-\zeta_{\tau}^{-}(\tau+\si)-X(\tau+\si)).
\end{equation}

\textbf{Case $\rm(ii)$. $x-\zeta-X(\tau)\leq-M$.} We write
\begin{equation}\label{middle-region}
\begin{split}
&u(\tau+\si,x;t_{0})-u(\tau+\si,x-\zeta_{\tau}^{-}(\tau+\si)-\tilde{\ep}\tilde{\de})\\
&\quad\quad=\big[u(\tau+\si,x;t_{0})-u(\tau+\si,x-\zeta_{\tau}^{-}(\tau+\si))\big]\\
&\quad\quad\quad+\big[u(\tau+\si,x-\zeta_{\tau}^{-}(\tau+\si))-u(\tau+\si,x-\zeta_{\tau}^{-}(\tau+\si)-\tilde{\ep}\tilde{\de})\big]\\
&\quad\quad\geq-q_{\tau}(\tau+\si)\Ga(x-\zeta_{\tau}^{-}(\tau+\si)-X(\tau+\si))\\
&\quad\quad\quad+\big[u(\tau+\si,x-\zeta_{\tau}^{-}(\tau+\si))-u(\tau+\si,x-\zeta_{\tau}^{-}(\tau+\si)-\tilde{\ep}\tilde{\de})\big],
\end{split}
\end{equation}
where we used the first inequality in \eqref{known-estimate}. For the term in the bracket, we first choose $M=M(\al)$ such that $-M+\frac{A\ep}{\om}\leq-\tilde{h}_{-}$, where $\tilde{h}_{-}$ is as in Lemma \ref{lem-decaying-space-derivative}. Then, we have
\begin{equation*}
\begin{split}
x-\zeta_{\tau}^{-}(\tau+\si)-X(\tau+\si)&\leq x-\zeta_{\tau}^{-}(\tau+\si)-X(\tau)\\
&=x-\zeta-X(\tau)+\frac{A\ep}{\om}(1-e^{-\om\si})\\
&\leq-M+\frac{A\ep}{\om}\leq-\tilde{h}_{-}.
\end{split}
\end{equation*}
It then follows from Lemma \ref{lem-decaying-space-derivative} that
\begin{equation*}
\begin{split}
&u(\tau+\si,x-\zeta_{\tau}^{-}(\tau+\si))-u(\tau+\si,x-\zeta_{\tau}^{-}(\tau+\si)-\tilde{\ep}\tilde{\de})\\
&\quad\quad=u_{x}(\tau+\si,x-\zeta_{\tau}^{-}(\tau+\si)-x_{*})\tilde{\ep}\tilde{\de}\quad(\text{where}\,\,x_{*}\in[0,\tilde{\ep}\tilde{\de}]\subset[0,1])\\
&\quad\quad\geq-\tilde{C}_{-}e^{\tilde{c}_{-}(x-\zeta_{\tau}^{-}(\tau+\si)-x_{*}-X(\tau+\si)+\tilde{h}_{-})}\tilde{\ep}\tilde{\de}\\
&\quad\quad=-\tilde{C}_{-}e^{\tilde{c}_{-}(x-\zeta_{\tau}^{-}(\tau+\si)-x_{*}-X(\tau)+\tilde{h}_{-})}e^{-\tilde{c}_{-}(X(\tau+\si)-X(\tau))}\tilde{\ep}\tilde{\de}\\
&\quad\quad\geq-\tilde{C}_{-}e^{-\tilde{c}_{-}(X(\tau+\si)-X(\tau))}\tilde{\ep}\tilde{\de}\\
&\quad\quad\geq-\tilde{C}_{-}e^{-\tilde{c}_{-}c_{\min}\si}\tilde{\ep}\tilde{\de}.
\end{split}
\end{equation*}
Going back to \eqref{middle-region}, we find
\begin{equation}\label{middle-region-12345}
\begin{split}
&u(\tau+\si,x;t_{0})-u(\tau+\si,x-\zeta_{\tau}^{-}(\tau+\si)-\tilde{\ep}\tilde{\de})\\
&\quad\quad\geq-q_{\tau}(\tau+\si)\Ga(x-\zeta_{\tau}^{-}(\tau+\si)-X(\tau+\si))-\tilde{C}_{-}e^{-\tilde{c}_{-}c_{\min}\si}\tilde{\ep}\tilde{\de}\\
&\quad\quad=-[q_{\tau}(\tau+\si)+\tilde{C}_{-}e^{-\tilde{c}_{-}c_{\min}\si}\tilde{\ep}\tilde{\de}\big]\Ga(x-\zeta_{\tau}^{-}(\tau+\si)-X(\tau+\si))\\
&\quad\quad\geq-[q_{\tau}(\tau+\si)+\tilde{C}_{-}\tilde{\ep}\tilde{\de}\big]\Ga(x-\zeta_{\tau}^{-}(\tau+\si)-X(\tau+\si))
\end{split}
\end{equation}
if we choose $M$ large so that $-M+\frac{A\ep}{\om}\leq-M_{1}-1$, and hence, $\Ga(x-\zeta_{\tau}^{-}(\tau+\si)-X(\tau+\si))=1$.

\textbf{Case $\rm(iii)$. $x-\zeta-X(\tau)\geq M$.} Choosing $M=M(\al,\si)$ larger, say $M-c_{\max}\si\geq\max\{M_{1}+1,\tilde{h}_{+}+1\}$, we have
\begin{equation*}
\begin{split}
x-\zeta_{\tau}^{-}(\tau+\si)-X(\tau+\si)&=x-\zeta-X(\tau)+\frac{A\ep}{\om}(1-e^{-\om\si})-(X(\tau+\si)-X(\tau))\\
&\geq M-c_{\max}\si\geq\max\{M_{1}+1,\tilde{h}_{+}+1\}.
\end{split}
\end{equation*}
As a result,
\begin{equation*}
\Ga(x-\zeta_{\tau}^{-}(\tau+\si)-X(\tau+\si))=e^{-\al(x-\zeta_{\tau}^{-}(\tau+\si)-X(\tau+\si)-M_{1})}
\end{equation*}
and by Lemma \ref{lem-decaying-space-derivative}
\begin{equation*}
u_{x}(\tau+\si,x-\zeta_{\tau}^{-}(\tau+\si)-x_{*})\geq\tilde{C}_{+}e^{-\tilde{c}_{+}(x-\zeta_{\tau}^{-}(\tau+\si)-x_{*}-X(\tau+\si)-\tilde{h}_{+})},\quad \forall x_{*}\in[0,1].
\end{equation*}
Together with the first inequality in \eqref{known-estimate}, we deduce
\begin{equation*}
\begin{split}
&u(\tau+\si,x;t_{0})-u(\tau+\si,x-\zeta_{\tau}^{-}(\tau+\si)-\tilde{\ep}\tilde{\de})\\
&\quad\quad\geq u(\tau+\si,x-\zeta_{\tau}^{-}(\tau+\si))-u(\tau+\si,x-\zeta_{\tau}^{-}(\tau+\si)-\tilde{\ep}\tilde{\de})\\
&\quad\quad-q_{\tau}(\tau+\si)\Ga(x-\zeta_{\tau}^{-}(\tau+\si)-X(\tau+\si))\\
&\quad\quad=u_{x}(\tau+\si,x-\zeta_{\tau}^{-}(\tau+\si)-x_{*})\tilde{\ep}\tilde{\de}-q_{\tau}(\tau+\si)e^{-\al(x-\zeta_{\tau}^{-}(\tau+\si)-X(\tau+\si)-M_{1})}\\
&\quad\quad\geq-\tilde{C}_{+}e^{-\tilde{c}_{+}(x-\zeta_{\tau}^{-}(\tau+\si)-x_{*}-X(\tau+\si)-\tilde{h}_{+})}\tilde{\ep}\tilde{\de}-q_{\tau}(\tau+\si)e^{-\al(x-\zeta_{\tau}^{-}(\tau+\si)-X(\tau+\si)-M_{1})}\\
&\quad\quad\geq-[\tilde{C}_{+}\tilde{\ep}\tilde{\de}+q_{\tau}(\tau+\si)]e^{-\al(x-\zeta_{\tau}^{-}(\tau+\si)-X(\tau+\si)-M_{1})},
\end{split}
\end{equation*}
if we choosing $\al$ smaller so that $\al\leq\tilde{c}_{+}$. Hence,
\begin{equation}\label{right-region-12345}
\begin{split}
&u(\tau+\si,x;t_{0})-u(\tau+\si,x-\zeta_{\tau}^{-}(\tau+\si)-\tilde{\ep}\tilde{\de})\\
&\quad\quad\geq-[\tilde{C}_{+}\tilde{\ep}\tilde{\de}+q_{\tau}(\tau+\si)]\Ga(x-\zeta_{\tau}^{-}(\tau+\si)-X(\tau+\si)).
\end{split}
\end{equation}

Thus, combining \eqref{left-region-12345}, \eqref{middle-region-12345}, \eqref{right-region-12345} and the second inequality in \eqref{known-estimate}, we find
\begin{equation}\label{lmr-regions-1}
\begin{split}
&u(\tau+\si,x-\zeta_{\tau}^{-}(\tau+\si)-\tilde{\ep}\tilde{\de})-\bar{q}_{\tau}(\tau+\si,\tilde{\ep})\Ga(x-\zeta_{\tau}^{-}(\tau+\si)-X(\tau+\si))\\
&\quad\quad\leq u(\tau+\si,x;t_{0})\\
&\quad\quad\quad\leq u(\tau+\si,x-\zeta_{\tau}^{+}(\tau+\si)-\de)+q_{\tau}(\tau+\si)\Ga(x-\zeta_{\tau}^{+}(\tau+\si)-\de-X(\tau+\si)),
\end{split}
\end{equation}
where
\begin{equation*}
\bar{q}_{\tau}(\tau+\si,\tilde{\ep})=
\begin{cases}
q_{\tau}(\tau+\si),\quad&x-\zeta-X(\tau)\in[-M,M],\\
q_{\tau}(\tau+\si)+\tilde{C}_{-}\tilde{\ep}\tilde{\de},\quad&x-\zeta-X(\tau)\leq-M,\\
\tilde{C}_{+}\tilde{\ep}\tilde{\de}+q_{\tau}(\tau+\si),\quad&x-\zeta-X(\tau)\geq M.
\end{cases}
\end{equation*}
Observe that the first $\Ga$ in \eqref{lmr-regions-1} is not in its right form, but from the monotonicity, we see
\begin{equation*}
\Ga(x-\zeta_{\tau}^{-}(\tau+\si)-X(\tau+\si))\leq\Ga(x-\zeta_{\tau}^{-}(\tau+\si)-\tilde{\ep}\tilde{\de}-X(\tau+\si)).
\end{equation*}
Since clearly $q_{\tau}(\tau+\si)\leq\bar{q}_{\tau}(\tau+\si,\tilde{\ep})$, we conclude from \eqref{lmr-regions-1} that
\begin{equation}\label{lmr-regions-2}
\begin{split}
&u(\tau+\si,x-\zeta_{\tau}^{-}(\tau+\si)-\tilde{\ep}\tilde{\de})-\bar{q}(\si)\Ga(x-\zeta_{\tau}^{-}(\tau+\si)-\tilde{\ep}\tilde{\de}-X(\tau+\si))\\
&\quad\quad\leq u(\tau+\si,x;t_{0})\\
&\quad\quad\quad\leq u(\tau+\si,x-\zeta_{\tau}^{+}(\tau+\si)-\de)+\bar{q}(\si)\Ga(x-\zeta_{\tau}^{+}(\tau+\si)-\de-X(\tau+\si)),
\end{split}
\end{equation}
where $\bar{q}(\si,\tilde{\ep})=\bar{q}_{\tau}(\tau+\si,\tilde{\ep})$ is independent of $\tau$. To apply Theorem \ref{thm-stability}$\rm(1)$, we choose $\si=\si(\al)$ sufficiently large and $\tilde{\ep}=\tilde{\ep}(\si,M,\ep_{0})=\tilde{\ep}(\al,\ep_{0})$ sufficiently small so that
\begin{equation*}
e^{-\om\si}\leq\frac{1}{2}\quad\text{and}\quad (\tilde{C}_{-}+\tilde{C}_{+})\tilde{\ep}\leq\frac{\ep_{0}}{2}.
\end{equation*}
Of course, for $\tilde{\ep}$, we should also take \eqref{aux-ep} into consideration. As a result $\bar{q}(\si,\tilde{\ep})\leq\ep_{0}$. We then apply Theorem  \ref{thm-stability}(1) to \eqref{lmr-regions-2} to conclude that
\begin{equation}\label{lmr-regions-3}
\begin{split}
&u(t,x-\zeta^{-}(t))-\bar{q}(\si,\tilde{\ep})e^{-\om(t-\tau-\si)}\Ga(x-\zeta^{-}(t)-X(t))\\
&\quad\quad\leq u(t,x;t_{0})\leq u(t,x-\zeta^{+}(t))+\bar{q}(\si,\tilde{\ep}) e^{-\om(t-\tau-\si)}\Ga(x-\zeta^{+}(t)-X(t))
\end{split}
\end{equation}
for $t\geq \tau+\si$, where
\begin{equation*}
\begin{split}
\zeta^{-}(t)&=\zeta_{\tau}^{-}(\tau+\si)+\tilde{\ep}\tilde{\de}-\frac{A\ep}{\om}(1-e^{-\om(t-\tau-\si)})=\zeta-\frac{2A\ep}{\om}+\tilde{\ep}\tilde{\de}+\frac{A\ep}{\om}[e^{-\om\si}+e^{-\om(t-\tau-\si)}],\\
\zeta^{+}(t)&=\zeta_{\tau}^{+}(\tau+\si)+\de+\frac{A\ep}{\om}(1-e^{-\om(t-\tau-\si)})=\zeta+\frac{2A\ep}{\om}+\de-\frac{A\ep}{\om}[e^{-\om\si}+e^{-\om(t-\tau-\si)}].
\end{split}
\end{equation*}
Let $\tilde{C}=\tilde{C}_{-}+\tilde{C}_{+}$. Setting
\begin{equation*}
q(t)=(\frac{\ep}{2}+\tilde{C}\tilde{\ep}\tilde{\de})e^{-\om(t-\tau-\si)},\quad\zeta(t)=\zeta^{-}(t)\quad\text{and}\quad \de(t)=\de-\tilde{\ep}\tilde{\de}+\frac{4A\ep}{\om}-\frac{2A\ep}{\om}[e^{-\om\si}+e^{-\om(t-\tau-\si)}],
\end{equation*}
we can rewrite \eqref{lmr-regions-3} as
\begin{equation}\label{lmr-regions-4}
\begin{split}
&u(t,x-\zeta(t))-q(t)\Ga(x-\zeta(t)-X(t))\\
&\quad\quad\leq u(t,x;t_{0})\leq u(t,x-\zeta(t)-\de(t))-q(t)\Ga(x-\zeta(t)-\de(t)-X(t))
\end{split}
\end{equation}
for $t\geq\tau+\si$.

The estimate \eqref{lmr-regions-4} is established under the assumption \eqref{either-or-2}.  If \eqref{either-or-1} holds, then similar arguments lead also to \eqref{lmr-regions-4} with $q(t)$ and $\de(t)$ of the same form and
\begin{equation*}
\begin{split}
\zeta(t)&=\zeta-\frac{2A\ep}{\om}+\frac{A\ep}{\om}[e^{-\om\si}+e^{-\om(t-\tau-\si)}].
\end{split}
\end{equation*}
We just remark that the choice of $\si$ in this case is still independent of $\de$, which follows from the observation that replacing $\de$ by $\tilde{\de}$ at appropriate steps when estimating the lower bound for the term
\begin{equation*}
u(\tau+\si,x-\zeta_{\tau}^{+}(\tau+\si)-\de+\tilde{\ep}\tilde{\de})-u(\tau+\si,x;t_{0}).
\end{equation*}
The lemma then follows.
\end{proof}

Now, we prove Theorem \ref{thm-asymptotic-stability}.

\begin{proof}[Proof of Theorem \ref{thm-asymptotic-stability}]
By Theorem \ref{thm-stability}$\rm(2)$, we have
\begin{equation*}
u(t,x-\zeta^{-})-\ep e^{-\om(t-t_{0})}\Ga(x-\zeta^{-}-X(t))\leq u(t,x;t_{0},u_{0})\leq u(t,x-\zeta^{+})+\ep e^{-\om(t-t_{0})}\Ga(x-\zeta^{+}-X(t))
\end{equation*}
for all $t\geq t_{0}$. In particular,
\begin{equation}\label{step-initial}
\begin{split}
&u(t_{0}+T_{0},x-\zeta_{0})-q_{0}\Ga(x-\zeta_{0}-X(t_{0}+T_{0}))\\
&\quad\quad\leq u(t_{0}+T_{0},x;t_{0},u_{0})\leq u(t,x-\zeta_{0}-\de_{0})+q_{0}\Ga(x-\zeta_{0}-\de_{0}-X(t_{0}+T_{0})),
\end{split}
\end{equation}
where $\zeta_{0}=\zeta^{-}$, $\de_{0}=\zeta^{+}-\zeta^{-}$, $q_{0}=\ep e^{-\om T_{0}}$ and $T_{0}>0$ is chosen so that
\begin{equation}\label{a-condition-12345678}
\frac{4A}{\om}\ep_{0}e^{-\om T_{0}}\leq\frac{\tilde{\ep}}{2}.
\end{equation}
Here, we may assume, without loss of generality, that $\de_{0}>1$. We now use iteration arguments to reduce $\de_{0}$.

Let $T\geq T_{0}$. Applying Lemma \ref{lem-key-asymptotic-stability} to \eqref{step-initial}, we find at the moment $t_{0}+T_{0}+\si+T$,
\begin{equation}\label{step-1}
\begin{split}
&u(t_{0}+T_{0}+\si+T,x-\zeta_{1})-q_{1}\Ga(x-\zeta_{1}-X(t_{0}+T_{0}+\si+T))\\
&\quad\quad\leq u(t_{0}+T_{0}+\si+T,x;t_{0},u_{0})\\
&\quad\quad\quad\leq u(t_{0}+T_{0}+\si+T,x-\zeta_{1}-\de_{1})-q_{1}\Ga(x-\zeta_{1}-\de_{1}-X(t_{0}+T_{0}+\si+T)),
\end{split}
\end{equation}
where
\begin{equation*}
\begin{split}
\zeta_{1}&\in[\zeta_{0}-\frac{2Aq_{0}}{\om},\zeta_{0}+\tilde{\ep}\min\{1,\de_{0}\}]\subset[\zeta_{0}-\frac{2Aq_{0}}{\om},\zeta_{0}+\tilde{\ep}\}],\\
0\leq\de_{1}&\leq\de_{0}-\tilde{\ep}\min\{1,\de_{0}\}+\frac{4Aq_{0}}{\om}=\de_{0}-\tilde{\ep}+\frac{4Aq_{0}}{\om}\leq \de_{0}-\frac{\tilde{\ep}}{2},\\
0\leq q_{1}&\leq (\frac{q_{0}}{2}+\tilde{C}\tilde{\ep}\min\{1,\de_{0}\})e^{-\om T}\leq\ep_{0}e^{-\om T}.
\end{split}
\end{equation*}
If $\de_{1}\leq1$, we stop. Otherwise, applying Lemma \ref{lem-key-asymptotic-stability} to \eqref{step-1}, we find at the moment $t_{0}+T_{0}+2(\si+T)$,
\begin{equation}\label{step-2}
\begin{split}
&u(t_{0}+T_{0}+2(\si+T),x-\zeta_{2})-q_{2}\Ga(x-\zeta_{2}-X(t_{0}+T_{0}+2(\si+T)))\\
&\quad\quad\leq u(t_{0}+T_{0}+2(\si+T),x;t_{0},u_{0})\\
&\quad\quad\quad\leq u(t_{0}+T_{0}+2(\si+T),x-\zeta_{2}-\de_{2})-q_{2}\Ga(x-\zeta_{2}-\de_{2}-X(t_{0}+T_{0}+2(\si+T))),
\end{split}
\end{equation}
where
\begin{equation*}
\begin{split}
\zeta_{2}&\in[\zeta_{1}-\frac{2Aq_{1}}{\om},\zeta_{1}+\tilde{\ep}\min\{1,\de_{1}\}]\subset[\zeta_{1}-\frac{2Aq_{1}}{\om},\zeta_{1}+\tilde{\ep}\}],\\
0\leq\de_{2}&\leq\de_{1}-\tilde{\ep}\min\{1,\de_{1}\}+\frac{4Aq_{1}}{\om}\leq\de_{0}-2\tilde{\ep}+2\frac{4Aq_{1}}{\om}\leq\de_{0}-2\frac{\tilde{\ep}}{2},\\
0\leq q_{2}&\leq(\frac{q_{1}}{2}+\tilde{C}\tilde{\ep}\min\{1,\de_{1}\})e^{-\om T}\leq\ep_{0}e^{-\om T}.
\end{split}
\end{equation*}
If $\de_{2}\leq 1$, we stop. Otherwise, applying Lemma \ref{lem-key-asymptotic-stability} to \eqref{step-2}. Repeating this, if $\de_{1},\de_{2},\dots\de_{N-1}$ are all greater than one $1$, then we will have
\begin{equation}\label{step-N}
\begin{split}
&u(t_{0}+T_{0}+N(\si+T),x-\zeta_{N})-q_{N}\Ga(x-\zeta_{N}-X(t_{0}+T_{0}+N(\si+T)))\\
&\quad\quad\leq u(t_{0}+T_{0}+N(\si+T),x;t_{0},u_{0})\\
&\quad\quad\quad\leq u(t_{0}+T_{0}+N(\si+T),x-\zeta_{N}-\de_{N})-q_{N}\Ga(x-\zeta_{N}-\de_{N}-X(t_{0}+T_{0}+N(\si+T))),
\end{split}
\end{equation}
where
\begin{equation*}
\begin{split}
\zeta_{N}&\in[\zeta_{N-1}-\frac{2Aq_{N-1}}{\om},\zeta_{N-1}+\tilde{\ep}\min\{1,\de_{N-1}\}]\subset[\zeta_{N-1}-\frac{2Aq_{N-1}}{\om},\zeta_{N-1}+\tilde{\ep}\}],\\
0\leq\de_{N}&\leq\de_{N-1}-\tilde{\ep}\min\{1,\de_{N-1}\}+\frac{4Aq_{N-1}}{\om}\leq\de_{0}-N\frac{\tilde{\ep}}{2},\\
0\leq q_{N}&\leq(\frac{q_{N-1}}{2}+\tilde{C}\tilde{\ep}\min\{1,\de_{N-1}\})e^{-\om T}\leq\ep_{0}e^{-\om T}.
\end{split}
\end{equation*}
Observe that there must be an $N$ such that $\de_{N}\leq\de_{0}-N\frac{\tilde{\ep}}{2}\leq1$. We then stop here. Setting $\tilde{T}_{0}=T_{0}+N(\si+T)$, $\tilde{\zeta}_{0}=\zeta_{N}$, $\tilde{\de}_{0}=\de_{N}$ and $\tilde{q}_{0}=q_{N}$ in \eqref{step-N}, we have
\begin{equation}\label{iteration-restart}
\begin{split}
&u(t_{0}+\tilde{T}_{0},x-\tilde{\zeta}_{0})-\tilde{q}_{0}\Ga(x-\tilde{\zeta}_{0}-X(t_{0}+\tilde{T}_{0}))\\
&\quad\quad\leq u(t_{0}+\tilde{T}_{0},x;t_{0},u_{0})\\
&\quad\quad\quad\leq u(t_{0}+\tilde{T}_{0},x-\tilde{\zeta}_{0}-\tilde{\de}_{0})-\tilde{q}_{0}\Ga(x-\tilde{\zeta}_{0}-\tilde{\de}_{0}-X(t_{0}+\tilde{T}_{0})),
\end{split}
\end{equation}
where $\tilde{\de}_{0}\in[0,1]$.

We now apply the above iteration arguments to \eqref{iteration-restart}, as a new initial step, to conclude the result. Recall \eqref{a-condition-12345678}, $\tilde{q}_{0}\leq\ep_{0}e^{-\om T_{0}}$ and $\tilde{C}\tilde{\ep}\leq\frac{\ep_{0}}{2}$. We now choose $T$ so larger  that
\begin{equation*}
(\frac{\ep_{0}}{2}e^{-\om T_{0}}+\tilde{C}\tilde{\ep})e^{-\om T}\leq (1-\frac{\tilde{\ep}}{2})\ep_{0}e^{-\om T_{0}}.
\end{equation*}
Applying Lemma \ref{lem-key-asymptotic-stability} to \eqref{iteration-restart}, we find
\begin{equation}\label{iteration-restart-1}
\begin{split}
&u(t_{0}+\tilde{T}_{0}+\si+T,x-\tilde{\zeta}_{1})-\tilde{q}_{1}\Ga(x-\tilde{\zeta}_{1}-X(t_{0}+\tilde{T}_{0}+\si+T))\\
&\quad\quad\leq u(t_{0}+\tilde{T}_{0}+\si+T,x;t_{0},u_{0})\\
&\quad\quad\quad\leq u(t_{0}+\tilde{T}_{0}+\si+T,x-\tilde{\zeta}_{1}-\tilde{\de}_{1})-\tilde{q}_{1}\Ga(x-\tilde{\zeta}_{1}-\tilde{\de}_{1}-X(t_{0}+\tilde{T}_{0}+\si+T)),
\end{split}
\end{equation}
where
\begin{equation*}
\begin{split}
\tilde{\zeta}_{1}&\in[\tilde{\zeta}_{0}-\frac{2A\tilde{q}_{0}}{\om},\tilde{\zeta}_{0}+\tilde{\ep}\tilde{\de}_{0}],\\
0\leq\tilde{\de}_{1}&\leq\tilde{\de}_{0}-\tilde{\ep}\tilde{\de}_{0}+\frac{4A\tilde{q}_{0}}{\om}\leq1-\tilde{\ep}+\frac{\tilde{\ep}}{2}=1-\frac{\tilde{\ep}}{2},\\
0\leq \tilde{q}_{1}&\leq(\frac{\tilde{q}_{0}}{2}+\tilde{C}\tilde{\ep}\tilde{\de}_{0})e^{-\om T}\leq(\frac{\ep_{0}}{2}e^{-\om T}+\tilde{C}\tilde{\ep})e^{-\om T}\leq(1-\frac{\tilde{\ep}}{2})\ep_{0}e^{-\om T_{0}}.
\end{split}
\end{equation*}

Applying Lemma \ref{lem-key-asymptotic-stability} to \eqref{iteration-restart-1}, we find
\begin{equation*}
\begin{split}
&u(t_{0}+\tilde{T}_{0}+2(\si+T),x-\tilde{\zeta}_{2})-\tilde{q}_{2}\Ga(x-\tilde{\zeta}_{2}-X(t_{0}+\tilde{T}_{0}+2(\si+T)))\\
&\quad\quad\leq u(t_{0}+\tilde{T}_{0}+2(\si+T),x;t_{0},u_{0})\\
&\quad\quad\quad\leq u(t_{0}+\tilde{T}_{0}+2(\si+T),x-\tilde{\zeta}_{2}-\tilde{\de}_{2})-\tilde{q}_{2}\Ga(x-\tilde{\zeta}_{2}-\tilde{\de}_{2}-X(t_{0}+\tilde{T}_{0}+2(\si+T))),
\end{split}
\end{equation*}
where
\begin{equation*}
\begin{split}
\tilde{\zeta}_{2}&\in[\tilde{\zeta}_{1}-\frac{2A\tilde{q}_{1}}{\om},\tilde{\zeta}_{1}+\tilde{\ep}\tilde{\de}_{1}],\\
0\leq\tilde{\de}_{2}&\leq\tilde{\de}_{1}-\tilde{\ep}\tilde{\de}_{1}+\frac{4A\tilde{q}_{1}}{\om}\leq(1-\frac{\tilde{\ep}}{2})(1-\tilde{\ep}+\frac{4A}{\om}\ep_{0}e^{-\om T_{0}})\leq(1-\frac{\tilde{\ep}}{2})^{2},\\
0\leq \tilde{q}_{2}&\leq(\frac{\tilde{q}_{1}}{2}+\tilde{C}\tilde{\ep}\tilde{\de}_{1})e^{-\om T}\leq(1-\frac{\tilde{\ep}}{2})(\frac{\ep_{0}}{2}e^{-\om T}+\tilde{C}\tilde{\ep})e^{-\om T}\leq(1-\frac{\tilde{\ep}}{2})^{2}\ep_{0}e^{-\om T_{0}}.
\end{split}
\end{equation*}

Then, applying Lemma \ref{lem-key-asymptotic-stability} repeatedly, we find for $n\geq3$
\begin{equation*}
\begin{split}
&u(t_{0}+\tilde{T}_{0}+n(\si+T),x-\tilde{\zeta}_{n})-\tilde{q}_{n}\Ga(x-\tilde{\zeta}_{n}-X(t_{0}+\tilde{T}_{0}+n(\si+T)))\\
&\quad\quad\leq u(t_{0}+\tilde{T}_{0}+n(\si+T),x;t_{0},u_{0})\\
&\quad\quad\quad\leq u(t_{0}+\tilde{T}_{0}+n(\si+T),x-\tilde{\zeta}_{n}-\tilde{\de}_{n})-\tilde{q}_{n}\Ga(x-\tilde{\zeta}_{n}-\tilde{\de}_{n}-X(t_{0}+\tilde{T}_{0}+n(\si+T))),
\end{split}
\end{equation*}
where
\begin{equation*}
\begin{split}
\tilde{\zeta}_{n}&\in[\tilde{\zeta}_{n-1}-\frac{2A\tilde{q}_{n-1}}{\om},\tilde{\zeta}_{n-1}+\tilde{\ep}\tilde{\de}_{n-1}],\\
0\leq\tilde{\de}_{n}&\leq\tilde{\de}_{n-1}-\tilde{\ep}\tilde{\de}_{n-1}+\frac{4A\tilde{q}_{n-1}}{\om}\leq(1-\frac{\tilde{\ep}}{2})^{n},\\
0\leq \tilde{q}_{n}&\leq(\frac{\tilde{q}_{n-1}}{2}+\tilde{C}\tilde{\ep}\tilde{\de}_{n-1})e^{-\om T}\leq(1-\frac{\tilde{\ep}}{2})^{n}\ep_{0}e^{-\om T_{0}}.
\end{split}
\end{equation*}
This clearly implies that $\tilde{\zeta}_{n}\to\tilde{\zeta}_{\infty}$, $\tilde{\de}_{n}\to0$ and $\tilde{q}_{n}\to0$ exponentially as $n\to \infty$, where $\tilde{\zeta}_{\infty}\in\R$. The theorem then follows readily.
\end{proof}

Finally, as a simple consequence of Theorem \ref{thm-asymptotic-stability} and Corollary \ref{cor-stability-12345}, we have

\begin{cor}\label{cor-asymptotic-stability-12345}
Let $u(t,x)$ and $X(t)$ be as in Proposition\ref{prop-property-approximating-sol}$\rm(2)$. Let $\beta_{0}>0$. Suppose $t_{0}\in\R$ and $\tilde{u}_{0}\in C_{\rm unif}^{b}(\R,\R)$ satisfy
\begin{equation*}
\begin{cases}
\tilde{u}_{0}:\R\to[0,1],\quad \liminf_{x\to-\infty}\tilde{u}_{0}(x)>\theta;\\
\exists C>0\,\,\text{s.t.}\,\,|\tilde{u}_{0}-u(t_{0},x)|\leq Ce^{-\beta_{0}(x-X(t_{0}))}\,\,\text{for}\,\,x\in\R.
\end{cases}
\end{equation*}
Then, there exist $C=C(u_{0})>0$, $\zeta_{*}=\zeta_{*}(u_{0})\in\R$ and $r=r(\beta_{0})>0$ such that
\begin{equation*}
\sup_{x\in\R}|u(t,x;t_{0},u_{0})-u(t,x-\zeta_{*})|\leq Ce^{-r(t-t_{0})}
\end{equation*}
for all $t\geq t_{0}$.
\end{cor}


\appendix

\section{Comparison principles}\label{sec-app-cp}

We state comparison principles used in the previous sections. See \cite[Proposition A.1]{ShSh14-2} for the proof.

\begin{prop}\label{prop-app-comparison}
Let $K:\R\times\R\to[0,\infty)$ be continuous and satisfy $\sup_{x\in\R}\int_{\R}K(x,y)dy<\infty$.
Let  $a:\R\times\R\to\R$ be continuous and uniformly bounded.

\begin{itemize}
\item[\rm(i)] Suppose that $X:[0,\infty)\to\R$ is continuous and that $u:[0,\infty)\times\R\to\R$ satisfies the following:   $u, u_t:[0,\infty)\times\R\to\R$
are continuous, the limit $\lim_{x\ra\infty}u(t,x)=0$ is locally uniformly in $t$,  and
\begin{equation*}
\begin{cases}
u_{t}(t,x)\geq \int_{\R}K(x,y)u(t,y)dy+a(t,x)u(t,x),\quad x>X(t),\,\,t>0,\\
u(t,x)\geq0,\quad x\leq X(t),\,\,t>0,\\
u(0,x)=u_{0}(x)\geq0,\quad x\in\R.
\end{cases}
\end{equation*}
Then $u(t,x)\geq0$ for $(t,x)\in(0,\infty)\times\R$.

\item[\rm(ii)]  Suppose that $X:[0,\infty)\to\R$ is continuous and that $u:[0,\infty)\times\R\to\R$ satisfies the following:   $u, u_t:[0,\infty)\times\R\to\R$
are continuous, the limit $\lim_{x\ra -\infty}u(t,x)=0$ is locally uniformly in $t$,  and
\begin{equation*}
\begin{cases}
u_{t}(t,x)\geq \int_{\R}K(x,y)u(t,y)dy+a(t,x)u(t,x),\quad x<X(t),\,\,t>0,\\
u(t,x)\geq0,\quad x\ge X(t),\,\,t>0,\\
u(0,x)=u_{0}(x)\geq0,\quad x\in\R.
\end{cases}
\end{equation*}
Then $u(t,x)\geq0$ for $(t,x)\in(0,\infty)\times\R$.

\item[\rm(iii)] Suppose  that $u:[0,\infty)\times\R\to\R$ satisfies the following:   $u, u_t:[0,\infty)\times\R\to\R$
is continuous, $\inf_{t\ge 0,x\in\R}u(t,x)>-\infty$, and
\begin{equation*}
\begin{cases}
u_{t}(t,x)\geq \int_{\R}K(x,y)u(t,y)dy+a(t,x)u(t,x),\quad x\in\R,\,\,t>0,\\
u(0,x)=u_{0}(x)\geq0,\quad x\in\R.
\end{cases}
\end{equation*}
Then $u(t,x)\geq0$ for $(t,x)\in(0,\infty)\times\R$. Moreover, if $u_0(x)\not\equiv 0$, then
$u(t,x)>0$ for $(t,x)\in(0,\infty)\times\R$.
\end{itemize}
\end{prop}


\bibliographystyle{amsplain}

\end{document}